\documentclass[12 pt, reqno]{amsart} 
\usepackage[utf8]{inputenc}
\usepackage{hyperref}
\usepackage{scalerel}
\usepackage{amsthm,amsmath,amsfonts,amssymb,mathrsfs}
\usepackage{mathtools} %for \mathrlap, to avoid powers taking space
\usepackage{subcaption} %double figure
\usepackage{bbm}
\usepackage{xcolor}
\usepackage{float}
\usepackage[margin=1in]{geometry}
\usepackage{multirow} % double line for tables

\usepackage{ulem} 

\numberwithin{equation}{section}

\newtheorem{theorem}{Theorem}[section]

\newtheorem{Def}{Definition}[section]
\newtheorem{corollary}[theorem]{Corollary}
\newtheorem{prop}[theorem]{Proposition}

\theoremstyle{definition}
\newtheorem{example}{Example}[section]

\newcommand{\R}{\mathbb{R}}
\newcommand{\Rplus}{\mathbb{R}_+}
\newcommand{\Rminus}{\mathbb{R}_{\scalebox{0.6}{$-$}}}

\newcommand{\parti}{\mathcal{S}}

\newcommand{\V}{\mathbb{V}}
\newcommand{\C}{\mathscr{C}}

\newcommand{\cop}{\mathcal{C}}

\newcommand{\tot}{\mathrm{tot}}

\newcommand{\I}{\mathcal{I}}

\def\mathTitle#1{
	\smallskip
	\noindent \textbf{#1}
	\smallskip
	\hfill\\
}

\def\probClass#1{\mathcal{P}(#1)}
\def\weightClass#1{\mathcal{W}\left(#1\right)}
\newcommand{\dirichDom}{{\V_{n-1}}}
\newcommand{\dirichDomTilde}{{\V_n}}

\newcommand{\Esp}{\mathbb{E}}
\def\Var{\mathop{\rm Var}\nolimits}

\newcommand{\br}{\,|\,}
\newcommand{\brr}{\,\bigr|\,}
\newcommand{\veps}{\varepsilon}

\newcommand{\dsp}{\displaystyle}
\newcommand{\dirich}{\Pi_{\alpha}}
\newcommand{\dirichTilde}{\Pi_{\Tilde{\alpha}}}
\newcommand{\dirichInv}{\Pi_{\alpha,\alpha_0}^{\mathrm{inv}}}
\newcommand{\uno}{\mathbf{1}}
\newcommand{\tn}{t_{-n}}

\newcommand{\tildx}{\Tilde{x}}

\newcommand{\mutild}{ \Tilde{\mu} }
\newcommand{\wtild}{{\Tilde{\Weight}}}
\newcommand{\rhotild}{ \Tilde{\rho} }
\newcommand{\muR}{\mu_{\mathrm{rad}}}
\newcommand{\nuR}{\nu_{\mathrm{rad}}}
\newcommand{\wR}{w_{\mathrm{rad}}}

\def\FX#1{F_{#1}}
\def\FtilX#1{\Tilde{F}_{#1}}

\newcommand{\matClass}{\mathcal{M}}
\newcommand{\matClassR}{\mathcal{M}_R}

\newcommand{\Weight}{{W}}
\newcommand{\DGSMk}{ \upsilon_{I_k,\Weight_{I_k}} }
\newcommand{\SKtot}{ S_{I_k}^\tot }

\def\norm#1{\left\| #1 \right\|}

\def\abs#1{\left| #1 \right|}
\def\nset#1#2{ \{#1,\dots,#2\} }
\def\mlap#1{  {\mathrlap{#1}} }

\DeclareMathOperator{\Hess}{Hess}
\DeclareMathOperator{\Jac}{Jac}

\DeclareMathOperator{\Lip}{Lip}

\DeclareMathOperator{\diag}{diag}
\newcommand{\Id}{\mathrm{Id}}

\newcommand{\normal}{\mathcal{N}}

\newcommand{\ie}{\textit{i.e.}\,}
\newcommand{\eg}{\textit{e.g.}\,}

\newcommand\blfootnote[1]{%
	\begingroup
	\renewcommand\thefootnote{}\footnote{#1}%
	\addtocounter{footnote}{-1}%
	\endgroup
}

\title[]{On weighted Poincar\'e inequalities for multivariate Liouville distributions - Application to Global Sensitivity Analysis}

\author{David Heredia}
\date{}

\begin{document}
	
	\maketitle
	
	\begin{abstract}
		In this work we establish weighted Poincaré inequalities for multivariate Liouville distributions, which are a generalization of the Dirichlet distribution. We also consider continuous elliptically contoured distributions, whose density levels are unions of hyperellipsoids. Our approach is based on a transport argument which allows weighted Poincar\'e inequalities to be transferred between probability measures. We apply our results to global sensitivity analysis and illustrate their practical use in a flood model case study, where the structure of dependence of the input variables is encoded by classical copulas.
	\end{abstract}
	
	
	\section{Introduction}
	Poincar\'e inequalities for probability distributions constitute a central topic in probability and functional analysis, with deep connections to concentration of measure, isoperimetry and spectral theory. Recall that the principle of a Poincar\'e inequality is to control the variance of functions in a functional space, typically a Sobolev space, by means of the integral of the Euclidean norm of their gradient. This provides a natural tool for applications in settings where derivative information is available.
	
	For instance, Poincar\'e inequalities have applications in dimension reduction via ridge approximation \cite{active_subspaces,zahm2019,verdiere2025}, where they are used to control the error induced by the reduction procedure, and have become a classical tool in Global Sensitivity Analysis (GSA) \cite{SobolKucherenko2009,lamboni,poincareintervals}. In GSA, one aims at quantifying the influence of input parameters, supposed to be modelled as random variables, in the output of a computationally expensive black-box model. When the input variables are independent, one-dimensional Poincar\'e inequalities provide the link between two commonly used sensitivity indices to quantify uncertainty: Sobol indices, which are highly interpretable but expensive in terms of computational resources, and Derivative based Global Sensitivity Measures (DGSM), which rely on the gradient information. This connection allows one to use DGSM as efficient tools for screening purposes, in order to identify variables with negligible influence.
	
	More recently, attention has shifted towards  weighted Poincar\'e inequalities. These are similar to the classical ones, but the norm of the gradient in the right-hand side integral is replaced by a norm induced by a matrix-valued function. There are at least two main motivations for considering such weighted inequalities. First, some probability measures do not satisfy a classical Poincar\'e inequality, such as heavy-tailed distributions (see \cite{BGL}). The second motivation is accuracy: the introduction of weights provides an additional degree of freedom that can significantly improve the accuracy of numerical results. For instance, in the context of GSA, \cite{Song,HerediaWeightPoincare} propose weighted DGSM indices, which generalize the classical ones and arise naturally from the application of one-dimensional weighted Poincar\'e inequalities. Although these results are limited to the assumption of independent input variables, they can naturally be extended to the case of block-wise independent variables using multidimensional weighted Poincar\'e inequalities. A key challenge for such extension is that these multidimensional inequalities remain far less understood than their one-dimensional counterparts, and existing results with practical relevance are only available for a limited range of probability measures. 
	
	The aim of this work is therefore twofold: to establish new results on multidimensional weighted Poincar\'e inequalities and to show their practical application in GSA. Our main focus is on the class of multivariate Liouville distributions and we additionally consider elliptically contoured distributions. These two families of probability distributions can be seen as generalizations of the Dirichlet and multivariate normal distributions, respectively, and both play a central role in statistics and related fields (see, for instance, \cite{symmetric_multivariate_distributions,elliptical_distributions_in_statistics,robust_statistics} and \cite{multivariate_liouville_3,inequalities_book}).
	In addition, by adopting a copula-based perspective, we are able to extend this type of results to more general distributions with prescribed copulas, which are of significant practical interest as they frequently arise in statistical modeling. To obtain these results, our approach is mainly based on a transport argument which, roughly speaking, allows weighted Poincar\'e inequalities to be transferred between probability measures while preserving the optimal Poincar\'e constant. This approach is direct and comparatively simple in contrast to, for instance, other 
	classical techniques relying on the spectral interpretation of weighted Poincar\'e inequalities (see \cite{BGL}). It is also naturally suited to measures characterized through transport, such as the moment measures addressed in \cite{cui2024optimal}.
	
	Our contributions and the organization of the paper are summarized as follows. Section \ref{sec: preliminares} introduces the preliminaries. Section \ref{sec: multivariate Liouville} is devoted to our main results on multivariate Liouville distributions. Its central result, Theorem \ref{theorem: Liouville measures Poincare inequality}, establishes weighted Poincar\'e inequalities for this class of distributions. Our approach is based on a radial-type decomposition which reduces the analysis to a product measure. As an application of this result, and under additional log-concavity assumptions, we further obtain two general Poincar\'e inequalities. Finally, another inequality is established in the same log-concave setting, using a direct argument through the famous Brascamp-Lieb inequality. Examples of application are presented, including weighted Poincar\'e inequalities for a family of heavy-tailed distributions.
	\\
	In Section \ref{sec: Weighted Poincare inequalities for elliptically contoured distributions} we make a brief digression to elliptically contoured distributions. These distributions are related to their spherical counterparts through a linear transport map. Proposition \ref{prop: elliptical measures Poincare inequality} then allows to transfer existing results for spherically contoured distributions, such as those emphasized in \cite{joulin_bonnefont_spherically_symmetric}, to the elliptical case. Examples are provided to illustrate this approach.
	\\
	Section \ref{sec: Measures with prescribed copulas} is motivated by practical applications. More precisely, we aim at extending our results to more general probability distributions for which the structure of dependence is encoded by a given copula. In Proposition \ref{prop: copula Poincare inequality} we show how to transfer a weighted Poincar\'e inequality for a given probability measure to any other measure sharing the same copula. Combined with the results emphasized above, this entails inequalities for distributions with Clayton or elliptical copulas, the latter including the Gaussian case.
	For measures associated with Gaussian copulas, we additionally establish classical Poincar\'e inequalities under uniform log-concavity assumptions on the marginals. 
	\\
	Finally, Section \ref{section: Application to Global Sensitivity Analysis} is devoted to applications in GSA. We introduce Sobol indices and weighted DGSM in the context of models with block-wise independent input variables. Then Proposition \ref{prop: upper bound Sobol DGSM} provides upper bounds on Sobol indices in terms of DGSM, obtained by applying weighted Poincaré inequalities. We illustrate these bounds numerically through a flood model case study, whose input variables are block-wise independent and coupled according to Gaussian or Clayton copulas. In the Gaussian copula setting we compare our numerical results with those obtained using classical Poincar\'e inequalities in \cite{steiner}, reporting more accurate estimates with our approach.
	
	\section{Preliminaries}
	\label{sec: preliminares}
	Let $\Omega\subset \R^n$ be a connected open set with a piecewise $\C^1$ boundary. We denote  $\probClass{\Omega}$ the set of probability measures $\mu$ on $\Omega$ that admit a density with respect to the Lebesgue measure, which is positive in $\Omega$. Denote $\matClass^n$ the set of symmetric positive definite matrices and let $\weightClass{\Omega}$ be the set of matrix-valued functions $\Weight\colon \Omega\rightarrow \matClass^n$. We systematically refer to functions $\Weight\in \weightClass{\Omega}$ as weights. Let $\langle\,\cdot\, , \, \cdot \rangle$ be the Euclidean inner product, with associated norm $\norm{\,\cdot\,}$.
	For a matrix $A\in \matClass^n$, its induced norm is defined as $\norm{x}_A=\sqrt{\langle A x, x\rangle}$, for all $x\in \R^n$. We also write $\lambda_{\max}(A)$ for the largest eigenvalue of $A$ and $\lambda_{\min}(A)$ for the smallest eigenvalue.\\
	Let $L^2(\mu)$ be the space of square-integrable functions with respect to a probability measure $\mu\in \probClass{\Omega}$ and, given $f\in L^2(\mu)$, denote $\nabla f$ the gradient of weak partial derivatives of $f$. 
	We can now define a weighted Poincar\'e inequality.
	\begin{Def}
		Let $\Omega\subset \R^n$ be a connected open set
		with a piecewise $\C^1$ boundary and let $\mu \in \probClass{\Omega}$. We say that $\mu\in \probClass{\Omega}$ satisfies a weighted Poincar\'e inequality with weight $\Weight \in \weightClass{\Omega}$ and finite constant $C>0$ if 
		
		\begin{equation}
			\label{eq: Poincare inequality}
			\Var_\mu(f):=\int_\Omega f^2\,d\mu-\left(\int_\Omega f\,d\mu\right)^2
			\leq C \int_\Omega 
			\norm{\nabla f}^2_{\Weight} d\mu,
		\end{equation}
		for all $f\in L^2(\mu)$ for which the right-hand side of the inequality is finite.
	\end{Def}
	In the following, whenever a weighted Poincar\'e inequality is stated, it will be understood to hold for all such functions $f$, to avoid unnecessary repetition.
	
	We denote $C_P(\mu,\Weight)$ the optimal (the smallest) positive constant for which \eqref{eq: Poincare inequality} holds. In the particular case where $\Weight$ is the identity matrix $\Id$, inequality \eqref{eq: Poincare inequality} reduces to the classical Poincar\'e inequality, which is a fundamental and active research topic. See \eg \cite{BGL} for an accessible introduction to the subject. In this case we simply write $C_P(\mu,\Id)=C_P(\mu)$. For a general weight $W$, note that if $\lambda_{\max}(W)$, which depends on the space variable, is bounded in $\Omega$, then the weighted Poincar\'e inequality \eqref{eq: Poincare inequality} implies a classical one with optimal constant satisfying 
	\[C_P(\mu)\leq \sup_{x\in \Omega}\lambda_{\max}(W(x))\,C_P(\mu,\Weight).\]
	
	Throughout this section we establish several weighted Poincar\'e inequalities. Our approach mainly relies on a well-known transport argument, that	we recall for completeness. Given a probability measure $\nu\in \probClass{\Omega}$ and a mapping $T\colon \Omega\rightarrow T(\Omega)$, we denote $T\# \nu$ the image measure (also called the pushforward measure) of $\nu$ by $T$. Suppose that $\nu$ satisfies a weighted Poincar\'e inequality with some weight $W_\nu\in \weightClass{\Omega}$ and that $T$ is a diffeomorphism. Then the image measure $\mu=T\# \nu\in \probClass{T(\Omega)}$ satisfies the following weighted Poincar\'e inequality
	\begin{multline}
		\label{eq: weighted Poincare inequality by transport}
		\Var_{\mu}(f)
		=\int_{\Omega}(f\circ T)^2\,d\nu-\left(\int_{\Omega}f\circ T\,d\nu\right)^2 \\ \leq C_P(\nu,\Weight_\nu) \int_{\Omega} \norm{\Jac(T)^\top \nabla f}_{\Weight_\nu }^2\, d\nu
		= C_P(\nu,\Weight_\nu)\int_{T(\Omega)}\norm{\nabla f}^2_{\Weight_\mu}\,d\mu,
	\end{multline}
	where $\Weight_\mu$ is the weight defined by
	\[\Weight_\mu=(\Jac(T)\,\Weight_\nu\Jac(T)^\top)\circ T^{-1},\]
	with $\Jac(T)$ standing for the Jacobian matrix of $T$ and $\Jac(T)^\top$ its transpose. Moreover, the optimal constants 
	$C_P(\mu,W_\mu)$ and $C_P(\nu,W_\nu)$ coincide. Indeed, the inequality above already gives us $C_P(\mu,\Weight_\mu)\leq C_P(\nu,\Weight_\nu)$. The reverse inequality follows by exchanging the roles of $\mu$ and $\nu=T^{-1}_{\#} \mu$, and using the identity $\Jac(T^{-1})=\Jac(T)^{-1}\circ T^{-1}$.
	
	Beyond transport, additional tools are available in the log-concave setting. Assuming that $\Omega$ is convex, recall that a measure $\mu$
	is said to be strictly log-concave if its
	density can be expressed as $\rho=e^{-V}$, where the potential $V$ is a strictly convex function  in $\Omega$. In this context, a fundamental weighted Poincar\'e inequality is the
	Brascamp-Lieb inequality (see \cite{brascamp_lieb}), which states that
	\begin{equation}
		\label{eq: Brascamp-Lieb}
		\Var_\mu(f)
		\leq \int_{\Omega} \norm{\nabla f}_{\Hess(V)^{-1}}^2\,d\mu,
	\end{equation}
	where $\Hess(V)$ denotes the Hessian matrix of $V$.
	In particular, when $\inf_{x\in \Omega}\lambda_{\min}(\Hess(V)(x))>0$, meaning that the measure $\mu$ is uniformly log-concave, it yields the well-known estimate
	\begin{equation}
		\label{eq: Bakry-Emery}
		C_P(\mu)\leq  \frac{1}{\inf_{x\in \Omega}\lambda_{\min}(\Hess(V)(x))},
	\end{equation}
	which is also an instance of the Bakry-Emery curvature dimension criterion (see  \cite{BGL}). This bound was later refined by Veysseire in \cite{veysseire}. We present his result as it appears in \cite{ABJ_intertwining}:
	\begin{equation}
		\label{eq: veysseire}
		C_P(\mu)\leq \int_{\Omega}\frac{1}{\lambda_{\min}(\Hess(V))} \,d\mu.
	\end{equation}
	This estimate will be applied in our forthcoming analysis.
	
	\section{Main results: weighted Poincar\'e inequalities for multivariate Liouville distributions}
	\label{sec: multivariate Liouville}
	Multivariate Liouville distributions arise in many areas of probability and statistics, notably in multivariate majorization \cite{inequalities_book} and in statistical reliability theory \cite{multivariate_liouville_3}. They also constitute  generalizations of the Dirichlet distribution, which itself has numerous connections in theoretical and applied fields (see \eg \cite{dirichlet_book}). 
	
	We say that a probability measure $\mu\in \probClass{\Rplus^n}$ in the strictly positive orthant $\Rplus^n=(0,\infty)^n$ is a multivariate Liouville distribution
	if its density function takes the form
	\begin{equation}
		\label{eq: multivariate Liouville distribution density}
		\rho(x)=Z^{-1}
		\prod_{i=1}^{n} x_i^{\alpha_i-1} g\left(\sum_{i=1}^{n} x_i\right),\quad x\in \Rplus^n,
	\end{equation}
	where $g\colon \Rplus\rightarrow \Rplus$ is an univariate positive function and $\alpha=(\alpha_1,\dots,\alpha_n)$ is a vector of positive parameters. In the following we denote $\abs{\alpha}=\sum_{i=1}^n \alpha_i$. We will systematically use the same notation $Z^{-1}$  for the normalization constant of any probability measure, when it is not needed  explicitly.
	
	Below we present some relevant examples of multivariate Liouville distributions, including those appearing in \cite{symmetric_multivariate_distributions,inequalities_book}.
	\begin{enumerate}
		\item The Dirichlet distribution $\dirich\in \probClass{\dirichDom}$, defined on the open simplex
		\begin{align*}
			\dirichDom &=\left\{ t_{-n}\colon= (t_1,\dots,t_{n-1})\in (0,1)^{n-1} \brr \sum_{i=1}^{n-1} t_i < 1 \right\},
		\end{align*}
		is the probability measure with density function
		\begin{equation}
			\label{eq: Dirichlet distribution density}
			\rho(\tn)=Z^{-1}
			\prod_{i=1}^{n-1}t_i^{\alpha_i-1}\left(1-\sum_{i=1}^{n-1}t_i\right)^\mlap{\alpha_n-1},\qquad \tn\in \dirichDom.
		\end{equation}
		This density is a particular case of \eqref{eq: multivariate Liouville distribution density} obtained by restricting it to $n-1$ variables and taking the function $g(s)=(1-s)^{\alpha_n-1}\mathbbm{1}_{s<1}$. The term inside the parenthesis in \eqref{eq: Dirichlet distribution density} should be interpreted as an additional component $t_n=1-\sum_{i=1}^{n-1} t_i$, so that $\sum_{i=1}^n t_i=1$.  The Dirichlet distribution $\dirich$ is a natural multivariate extension the Beta distribution, recovered when $n=2$.
		\item Another multivariate Liouville distribution is the inverted Dirichlet distribution.
		It is obtained 
		introducing an additional parameter $\alpha_0>0$ and taking the function
		$g(s)=
		(1+s)^{-(\abs{\alpha}+\alpha_0)}$, $s\in \Rplus$.
		The density is therefore defined as
		\begin{equation}
			\label{eq: Multivariate Liouville example inverted Dirichlet distribution}
			\rho(x)=Z^{-1}
			\prod_{i=1}^{n}x_i^{\alpha_i-1}\left(1+\sum_{i=1}^{n}x_i\right)^\mlap{-(\abs{\alpha}+\alpha_0)},\qquad 
			x\in \Rplus^n.
		\end{equation}
		We denote this measure by $\dirichInv$. The link between the Dirichlet distribution $\dirich$ and the inverted one $\dirichInv$ is given via transport, which we specify later in Example \ref{example: inverted Dirichlet} to establish a weighted Poincar\'e inequality for $\dirichInv$.
		\item As a particular case of the previous example, consider the multivariate Pareto distribution, also called multivariate Lomax distribution (see \cite{multivariate_pareto}), whose density function is given by
		\begin{equation}
			\label{eq: multivariate pareto}
			\rho(x)= 
			\frac{\Gamma(\gamma+n)}{\Gamma(\gamma)} 
			\left(1+\sum_{i=1}^n x_i\right)^\mlap{-(\gamma+n)},\qquad x\in \Rplus^n,   
		\end{equation}
		with $\gamma >0$ a parameter.
		\item Finally, consider the Gamma Liouville distribution,
		obtained when taking $g(s)=s^{\delta}e^{-s} 
		$, with $\delta>-\abs{\alpha}$. In this case the density in \eqref{eq: multivariate Liouville distribution density} becomes
		\begin{equation}
			\label{eq: Multivariate Liouville example correlated Gammas}
			\rho(x)=Z^{-1}
			\left(\sum_{i=1}^n x_i\right)^\delta\prod_{i=1}^n
			x_i^{\alpha_i-1}
			e^{-x_i},\quad x\in \Rplus^n.
		\end{equation}
		When $\delta=0$ this measure reduces to a product of independent gamma distributions.
	\end{enumerate}
	
	Now we are in position to establish weighted Poincar\'e inequalities for multivariate Liouville distributions. Our strategy is inspired by the approach used in \cite{bobkov_spherically_symmetric,joulin_bonnefont_spherically_symmetric} 
	for spherically contoured distributions  (\ie probability measures whose densities depend on $\norm{x}$, $x\in \R^n$). This approach relies on the fact that any spherical distribution $\nu\in \probClass{\R^n}$ can be expressed as an image of the product measure between a one-dimensional probability distribution
	$\nuR\in \probClass{\Rplus}$, encoding the radial component, and the uniform distribution on the $n$-sphere. This decomposition allows one to deal with the radial and angular components independently, using their corresponding Poincar\'e inequalities.
	
	In the same spirit, we represent any multivariate Liouville distribution $\mu\in \mathcal{P}(\Rplus^n)$ as an image of the product measure $\mu_S\otimes \dirich$, where $\mu_S\in \probClass{\Rplus}$ is a one-dimensional probability measure and $\dirich$ is the Dirichlet distribution. As similarly shown in \cite{joulin_bonnefont_spherically_symmetric} for the spherical case, we prove that weighted Poincar\'e inequalities for $\mu_S$ induce corresponding inequalities for $\mu$. This yields the main result of the paper, stated in the theorem below.
	\begin{theorem}
		\label{theorem: Liouville measures Poincare inequality}
		Let $\mu\in \probClass{\Rplus^n}$ be a multivariate Liouville distribution, with density function as in \eqref{eq: multivariate Liouville distribution density}. Consider the one-dimensional probability measure $\mu_S\in \probClass{\Rplus}$ with density proportional to $s\in \Rplus\mapsto s^{\abs{\alpha}-1}\,g(s)$. Suppose that $\mu_S$ satisfies a weighted Poincar\'e inequality with weight $w_S\in \weightClass{\Rplus}$. Denote $s(x)=\sum_{i=1}^n x_i$, $x\in \Rplus^n$, and
		\begin{equation*}
			K=\int_{\Rplus}\frac{s^2}{w_S(s)}\,d\mu_S(s).
		\end{equation*}
		Then $\mu$ satisfies the following weighted Poincar\'e inequality
		\begin{equation}
			\label{eq: Liouville measures Poincare inequality theorem}
			\Var_\mu(f)
			\leq \max\left(C_P(\mu_S,w_S),\frac{K}{\abs{\alpha}}\right)
			\int_{\Rplus^n}
			\frac{w_S\left(s(x)\right)}{s(x)}\sum_{i=1}^{n} x_i\, \left(\frac{\partial f}{\partial x_i}(x)\right)^2\,
			d\mu(x).
		\end{equation}
	\end{theorem}
	\begin{proof}
		We begin by introducing the radial-type decomposition of the multivariate Liouville distribution $\mu$. Consider the mapping $T\colon \Rplus\times \dirichDom\rightarrow \Rplus^n$ defined as
		\begin{equation}
			\label{eq: transformation multivariate Liouville}
			T(s,\tn)=\left(s\, t_1,\dots,s\, t_{n-1},s\left(1-\sum_{i=1}^{n-1}t_i\right)\right), \quad  (s,\tn)\in \Rplus\times \dirichDom.
		\end{equation}
		Then we claim that $\mu=T\# (\mu_S\otimes \dirich)$. Indeed, for all $(s,\tn)\in \Rplus\times \dirichDom$, the determinant of the Jacobian of $T$ is given by $\det\left(\Jac (T)(s,\tn)\right)=-s^{n-1}$ (see \eg \cite{dirichlet_book}) and thus for every measurable, non-negative or bounded function $f\colon \Rplus^n\rightarrow \R$ it follows that
		\begin{align*}
			\int_{\Rplus^n} f&\,d\mu=\int_{\dirichDom}\int_{\Rplus} (f\circ T)(s,\tn)\,d\mu_S(s)\,d\dirich(\tn)\\
			&=\int_{\dirichDom} \int_{\Rplus} (f\circ T)(s,\tn)\,
			Z^{-1}
			\, \prod_{i=1}^{n-1} \left(s\,t_i\right)^{\alpha_i-1} \left(s\left(1-\sum_{i=1}^{n-1}t_i\right)\right)^{\alpha_n-1}s^{n-1}g(s)\,ds \,d\tn\\
			&=\int_{\dirichDom} \int_{\Rplus} (f\circ T)(s,\tn)\,
			Z^{-1}
			\, \prod_{i=1}^{n-1} t_i^{\alpha_i-1}\left(1-\sum_{i=1}^{n-1}t_i\right)^{\alpha_n-1} s^{\abs{\alpha}-1}g(s)\,ds \,d\tn\\
			&=\int_{\dirichDom} \int_{\Rplus} f\circ T\, d\mu_S\,d \dirich.
		\end{align*}
		
		Before proceeding to the proof computations, let us  introduce the relevant weighted Poincar\'e inequality satisfied by the Dirichlet distribution $\dirich$, which is a key ingredient in the proof. Consider the weight ${\Weight_\tau}\in \weightClass{\dirichDom}$ defined as
		\[\Weight_\tau(\tn)=\diag(\tn)-\tn \,\tn^\top,\quad \tn\in \dirichDom,\]
		where $\diag(\tn)$ denotes the diagonal matrix with $\tn$ on its diagonal. Using Cauchy-Schwarz' inequality we observe that it is indeed positive definite for all $\tn\in \dirichDom$: for any nonzero vector $b\in \R^{n-1}$ we have
		\begin{multline*}
			\langle {\Weight_\tau (\tn)} b,b\rangle=\sum_{i=1}^{n-1} t_i b_i^2-\left(\sum_{i=1}^{n-1} t_i b_i\right)^2\\
			\geq \sum_{i=1}^{n-1} t_i b_i^2-\left(\sum_{i=1}^{n-1} t_i\right)\left( \sum_{i=1}^{n-1} t_i b_i^2\right)=\left(1-\sum_{i=1}^{n-1} t_i\right)\left( \sum_{i=1}^{n-1} t_i b_i^2\right)> 0.
		\end{multline*}\\
		Then the Dirichlet distribution $\dirich$ satisfies a weighted Poincar\'e inequality with weight $\Weight_\tau$, meaning that
		\begin{align}
			\label{eq: Poincare inequality Dirichlet simplex}
			\scaleobj{.96}{
				\dsp
				\int_{\dirichDom} f^2 \,d\dirich\leq \frac{1}{\abs{\alpha}} \int_{\dirichDom} \norm{\nabla f}^2_{{\Weight_\tau}}\, d\dirich 
				= \frac{1}{\abs{\alpha}} \int_{\dirichDom} \left(\sum_{i=1}^{n-1} t_i \left(\frac{\partial f}{\partial t_i}\right)^2-\left(\sum_{i=1}^{n-1} t_i \frac{\partial f}{\partial t_i}\right)^2\right)\,d\dirich(\tn).}
		\end{align}
		Above, the factor $1/\abs{\alpha}$ is the optimal Poincar\'e constant $C_P(\dirich,\Weight_\tau)$, a fact that can be explained from a spectral point of view. Indeed, recall that $C_P(\dirich,\Weight_\tau)$ is characterized as the inverse of the spectral gap, \ie the first positive eigenvalue, of a self-adjoint operator associated with the weighted Poincar\'e inequality (see \cite{BGL}). In \cite{shimakura}, such a spectral gap is shown to be equal to $\abs{\alpha}$, and therefore $C_P(\dirich,\Weight_\tau)=1/\abs{\alpha}$.
		See also \cite{miclo_dirichlet} for an elegant alternative proof based on a stochastic representation of the Dirichlet distribution $\dirich$ in terms of gamma distributions.
		
		We prove now the weighted Poincar\'e inequality \eqref{eq: Liouville measures Poincare inequality theorem}. We have
		\begin{multline}
			\label{eq: Theorem proof both integrals together}
			\Var_\mu(f)=\int_{\dirichDom} \int_{\Rplus} (f\circ T)^2 d\mu_S\, d\dirich - \left(\int_{\dirichDom} \int_{\Rplus} f\circ T\, d\mu_S\, d\dirich\right)^2\\
			=\int_{\dirichDom} \underbrace{\int_{\Rplus} (f\circ T)^2 d\mu_S-\left(\int_{\Rplus} f\circ T \,d\mu_S\right)^2}_{\Var_{\mu_S}(f\circ T)}\, d\dirich+
			\int_{\dirichDom} h^2 \,d\dirich-\left(\int_{\dirichDom} h \,d\dirich\right)^2,
		\end{multline}
		where the function $h\colon \dirichDom\rightarrow \R$ is given by
		\[h(\tn)=\int_{\Rplus} f\circ T(s,\tn) \,d\mu_S(s), \quad \tn\in \dirichDom.\]
		We deal with the variance in \eqref{eq: Theorem proof both integrals together}. 
		Since $\mu_S$ satisfies a weighted Poincar\'e inequality with weight $w_S\in \weightClass{\Rplus}$, we have
		\begin{equation}
			\label{eq: Poincare inequality first integral mu_S multi-dim}
			\scaleobj{.97}{
				\dsp
				\Var_{\mu_S}(f\circ T)
				\leq C_P(\mu_S,w_S) 
				\int_{\Rplus} w_S(s) \left(\sum_{i=1}^{n-1} t_i\, \left(\frac{\partial f}{\partial x_i}\circ T\right)+\left(1-\sum_{i=1}^{n-1} t_i \right)\left(\frac{\partial f}{\partial x_ n}\circ T\right)\right)^2\,d\mu_S(s).}
		\end{equation}
		Next, to deal with the remaining terms in \eqref{eq: Theorem proof both integrals together} we apply the weighted Poincar\'e inequality for the Dirichlet distribution $\dirich$ given in \eqref{eq: Poincare inequality Dirichlet simplex} to the function $h$. We obtain
		\[
		\int_{\dirichDom} h^2 \,d\dirich-\left(\int_{\dirichDom} h \,d\dirich\right)^2
		\leq \frac{1}{\abs{\alpha}} \int_{\dirichDom} \norm{\int_{\Rplus} \nabla_{\tn} (f\circ T)\,d\mu_{S}}^2_{\Weight_\tau} \, d\dirich,\]
		where 
		$$\nabla_{\tn}(f\circ T)=\left(s\left(\frac{\partial f}{\partial x_i} - \frac{\partial f}{\partial x_ n}\right)\circ T\right)_{i=1}^{n-1},$$ and the integral with respect to $\mu_S$ is understood as a vector of coordinate-wise integrals. Introducing the norm $\norm{\,\cdot\,}_{W_\tau}$ inside the integral with respect to the probability measure $\mu_S$ and using then Cauchy-Schwarz' inequality it follows that
		\begin{multline}
			\label{eq: Poincare inequality second integral Pi_n-1 multi-dim}
			\scaleobj{.93}{\dsp
				\int_{\dirichDom} h^2 \,d\dirich-\left(\int_{\dirichDom} h \,d\dirich\right)^2 \leq \frac{K}{\abs{\alpha}} \int_{\dirichDom}\int_{\Rplus} \frac{w_S(s)}{s^2} \norm{ \nabla_{\tn} (f\circ T)}^2_{\Weight_\tau}\,d\mu_S(s) \,d\dirich(\tn)}\\
			\scaleobj{.93}{\dsp
				=\frac{K}{\abs{\alpha}} \int_{\dirichDom}\int_{\Rplus}
				w_S(s)\left(\sum_{i=1}^{n-1} t_i \left(
				\frac{\partial f}{\partial x_i}-\frac{\partial f}{\partial x_n}
				\right)^2 \circ T
				-\left(\sum_{i=1}^{n-1} t_i \left(\frac{\partial f}{\partial x_i}-\frac{\partial f}{\partial x_n}\right)\circ T \right)^2 \right)
				d\mu_S(s) \,d\dirich(\tn)},
		\end{multline}
		where we recall that $K=\int_{\Rplus} (s^2/w_S(s))\,d\mu_S(s)$.
		One can check that the quantity inside the big parenthesis can be rewritten as
		\begin{multline*}
			\scaleobj{.93}{\dsp    \sum_{i=1}^{n-1} t_i \left(\frac{\partial f}{\partial x_i}-\frac{\partial f}{\partial x_ n}\right)^2\circ T-\left(\sum_{i=1}^{n-1} t_i \left(\frac{\partial f}{\partial x_i}-\frac{\partial f}{\partial x_ n}\right)\circ T \right)^2}\\
			\scaleobj{.93}{\dsp=\sum_{i=1}^{n-1} t_i\, \left(\frac{\partial f}{\partial x_i}\circ T\right)^2+\left(1-\sum_{i=1}^{n-1}t_i\right)\left(\frac{\partial f}{\partial x_ n}\circ T\right)^2-\left(\sum_{i=1}^{n-1}t_i\, \left(\frac{\partial f}{\partial x_i}\circ T\right)+\left(1-\sum_{i=1}^{n-1}t_i\right)
				\left(\frac{\partial f}{\partial x_n}\circ T\right)\right)^2.}
		\end{multline*}
		But note then that the third term on the right-hand side is exactly the same one appearing in the first inequality \eqref{eq: Poincare inequality first integral mu_S multi-dim}.  Therefore, denoting $t_n := 1 - \sum_{i=1}^{n-1} t_i$ and combining inequalities \eqref{eq: Poincare inequality first integral mu_S multi-dim} and \eqref{eq: Poincare inequality second integral Pi_n-1 multi-dim} into \eqref{eq: Theorem proof both integrals together}, we obtain
		\begin{multline*}
			\Var_\mu(f) \leq C_P(\mu_S,w_S) \int_{\dirichDom}\int_{\Rplus} w_S(s) \left( \sum_{i=1}^n t_i \left(\frac{\partial f}{\partial x_i}\circ T\right) \right)^2 d\mu_S(s) \,d\dirich(\tn) \\
			+ \frac{K}{\abs{\alpha}} \int_{\dirichDom}\int_{\Rplus} w_S(s) \left( \sum_{i=1}^n t_i \left(\frac{\partial f}{\partial x_i}\circ T\right)^2 - \left( \sum_{i=1}^n t_i \,\left(\frac{\partial f}{\partial x_i}\circ T\right) \right)^2 \right) d\mu_S(s) \,d\dirich(\tn) \\
			\leq \max\left( C_P(\mu_S,w_S), \frac{K}{\abs{\alpha}} \right) \int_{\dirichDom}\int_{\Rplus} w_S(s) \sum_{i=1}^n t_i \left(\frac{\partial f}{\partial x_i}\circ T\right)^2 d\mu_S(s) \,d\dirich(\tn).
		\end{multline*}
		Finally, we come back to the original variables by applying the inverse transformation
		\begin{equation}
			\label{eq: inverse transformation multivariate Liouville}
			(s,\tn)=T^{-1}(x)=\left(\sum_{i=1}^n x_i,\frac{x_1}{\sum_{i=1}^n x_n},\dots,\frac{x_{n-1}}{\sum_{i=1}^n x_i}\right),\qquad x\in \Rplus^n,
		\end{equation}
		so that $t_n=1 - \sum_{i=1}^{n-1} t_i=x_n/\sum_{i=1}^n x_i$. This yields the desired inequality \eqref{eq: Liouville measures Poincare inequality theorem}, completing the proof.
	\end{proof}
	Before presenting examples of application of Theorem \ref{theorem: Liouville measures Poincare inequality} we first discuss its consequences in the log-concave setting, in analogy to the developments in the spherical case in \cite{joulin_bonnefont_spherically_symmetric}. In the following Corollary we establish two general weighted Poincar\'e inequalities under the assumptions that the function $g$ in \eqref{eq: multivariate Liouville distribution density} is log-concave and that $\abs{\alpha}>1$. These assumptions are not very restrictive. For instance, they admit the standard choice $\alpha_i=1$ for each $i\in \nset{1}{n}$ and are satisfied by all the examples introduced at the beginning of the section, under suitable choices on their additional parameters.
	\begin{corollary}
		\label{corollary: Liouville measures Poincare inequality}
		Let $\mu$ a multivariate Liouville distribution, with density function given in  \eqref{eq: multivariate Liouville distribution density}. Suppose that $g$ is log-concave (that is, $g=e^{-V_g}$ with $V_g$ convex) and 
		$\abs{\alpha}>1$. Denote $s(x)=\sum_{i=1}^n x_i$, $x\in \Rplus^n$. Then $\mu$ satisfies the following weighted Poincar\'e inequalities:
		\begin{equation}
			\label{eq: Liouville measures Poincare inequality with s^2 corollary}
			\Var_\mu(f)
			\leq \frac{1}{\abs{\alpha}-1} \int_{\Rplus^n}
			s(x)\,
			\sum_{i=1}^{n} x_i\, \left(\frac{\partial f}{\partial{x_i}}(x)\right)^2
			d\mu(x),
		\end{equation}
		and
		\begin{equation}
			\label{eq: Liouville measures classical Poincare inequality corollary}
			\Var_\mu(f)
			\leq \frac{1}{\abs{\alpha}-1} \int_{\Rplus^n} s(x)^2\,d\mu\times \int_{\Rplus^n}
			\sum_{i=1}^{n} \frac{x_i}{s(x)} \left(\frac{\partial f}{\partial{x_i}}(x)\right)^2\,
			d\mu(x).
		\end{equation}
		In particular, the latter leads to a classical Poincaré inequality after using the bound $x_i/s(x)\leq 1$ for each $i\in \nset{1}{n}$.
	\end{corollary}
	\begin{proof}
		The idea of the proof is to establish two different Poincar\'e inequalities for the one-dimensional probability measure $\mu_S$, taking advantage of the polynomial factor $s^{\abs{\alpha}-1}$ in its density which contributes to the log-concavity of the measure. Indeed, writing the density of $\mu_S$ proportional to $e^{-V_S}$, where the potential $V_S$ is given by
		\[V_S(s)=V_g(s)-(\abs{\alpha}-1)\log(s),\quad s\in \Rplus,\]
		then one observes that $V_S$ is convex since $V_g$ is, and moreover,
		\begin{equation*}
			\frac{1}{V_S''(s)}=\frac{1}{V_g''(s)+(\abs{\alpha}-1)\dfrac{1}{s^2}} \leq \frac{1}{\abs{\alpha}-1}s^2,\quad s\in \Rplus.
		\end{equation*}
		One can then apply the Brascamp-Lieb inequality \eqref{eq: Brascamp-Lieb}
		to the measure $\mu_S$, obtaining
		a weighted Poincar\'e inequality with weight $w_S(s)=s^2$ and optimal constant satisfying the inequality
		\[C_P(\mu_S,w_s)\leq \frac{1}{\abs{\alpha}-1}.\]
		With this choice of weight we have
		\[K=\int_{\Rplus}\frac{s^2}{w_S(s)}\,d\mu_S(s)=1\quad\mbox{and}\quad \max\left(C_P(\mu_S,w_S),\frac{K}{\abs{\alpha}}\right)\leq \frac{1}{\abs{\alpha}-1}.\]
		Therefore Theorem \ref{theorem: Liouville measures Poincare inequality} yields the first weighted Poincar\'e inequality of Corollary \ref{corollary: Liouville measures Poincare inequality}, given in \eqref{eq: Liouville measures Poincare inequality with s^2 corollary}. \\
		The second inequality \eqref{eq: Liouville measures classical Poincare inequality corollary} is obtained in a similar way, using Veysseire's inequality \eqref{eq: veysseire} instead of Brascamp-Lieb's, from which it follows that $\mu_S$ satisfies a classical Poincar\'e inequality (\ie with  weight $w_S(s)=1$) with optimal constant satisfying $$C_P(\mu_S)\leq \frac{1}{\abs{\alpha}-1}\int_{\Rplus}s^2\,d\mu_S(s).$$
		Then we have
		\[K=\int_{\Rplus} s^2 d\mu_S(s) \quad \mbox{and}\quad \max\left(C_P(\mu_S),\frac{K}{\abs{\alpha}}\right)\leq \frac{1}{\abs{\alpha}-1}\int_{\Rplus}s^2\,d\mu_S.\]
		Using the inverse transformation $T^{-1}$ in \eqref{eq: inverse transformation multivariate Liouville} we see that that the right-hand side integral is equal to the constant $\int_{\Rplus^n}s(x)^2\,d\mu(x)$ appearing in inequality \eqref{eq: Liouville measures classical Poincare inequality corollary}. This completes the proof.
	\end{proof}
	The two inequalities \eqref{eq: Liouville measures Poincare inequality with s^2 corollary} and \eqref{eq: Liouville measures classical Poincare inequality corollary} are not comparable. This can be seen, for instance, considering the function $f(x)=\sum_{i=1}^n x_i^{a}$ (with $a>0$). After some computations involving the transformation $T$ in \eqref{eq: transformation multivariate Liouville} one can show that, for this function, comparing the right-hand sides of the two inequalities reduces to the comparison between the terms
	\[\int_{\Rplus}s^{2a} \,d\mu_S(s)\quad \mbox{and}\quad  \int_{\Rplus}s^{2}d\mu_S(s)\times\int_{\Rplus}s^{2a-2} \,d\mu(s).\]
	Choosing $a$ and $\mu_S$ appropriately, one can find cases where either the first or the second quantity is smaller. 
	
	These two inequalities rely on the log-concavity of the function $g$. A natural question is whether this property implies the log-concavity of the measure $\mu$ itself, allowing for a direct general result via the Brascamp-Lieb inequality \eqref{eq: Brascamp-Lieb}. This is indeed the case, although it requires the assumption $\alpha_i>1$ for each $i\in \nset{1}{n}$, which is a stronger condition than  $\vert \alpha \vert > 1$, used in Corollary \ref{corollary: Liouville measures Poincare inequality}. Under this assumption, one obtains an additional weighted Poincar\'e inequality stated in the following proposition. 
	\begin{prop}
		Under the same notation and assumptions of Corollary \ref{corollary: Liouville measures Poincare inequality}, assume furthermore that $\alpha_i>1$ for each $i\in \nset{1}{n}$. Then the measure $\mu$ satisfies the following Poincar\'e inequality:
		\begin{equation}
			\label{eq: Liouville measures new Poincare inequality proposition}
			\Var_\mu(f)\leq \int_{\Rplus^n} \sum_{i=1}^n \frac{1}{\alpha_i-1}x_i^2\left(\frac{\partial f}{\partial x_i}(x)\right)^2 \,d\mu(x).
		\end{equation}
	\end{prop}
	\begin{proof}
		Under the additional assumption, the measure $\mu$ is strictly log-concave. Indeed, writing its density as $e^{-V}$, the potential $V$ is given by
		\[
		V(x) = \log(Z)
		- \sum_{i=1}^n (\alpha_i-1) \log (x_i) + V_g\left(\sum_{i=1}^n x_i\right),\quad x\in \Rplus^n,\]
		with Hessian matrix
		\[\Hess(V)(x)=\diag\left((\alpha_1-1)\frac{1}{x_1^2},\dots,(\alpha_n-1)  \frac{1}{x_n^2}\right)+V_g''\left(\sum_{i=1}^n x_i\right)\uno\uno^\top,\quad x\in \Rplus^n,\]
		where $\uno\in \R^n$ denotes the column vector with all entries equal to one. Since $V_g$ is convex, the second term is semi-positive definite, and hence the Hessian can be bounded from below (in the sense of symmetric matrices) by the first term alone. Therefore the Brascamp-Lieb inequality \eqref{eq: Brascamp-Lieb} entails \eqref{eq: Liouville measures new Poincare inequality proposition}.
	\end{proof}
	Inequality \eqref{eq: Liouville measures new Poincare inequality proposition} is not comparable to those in Corollary \ref{corollary: Liouville measures Poincare inequality}. For instance, in the first one given in \eqref{eq: Liouville measures Poincare inequality with s^2 corollary} the terms $s(x)\, x_i$ are larger than the $x_i^2$ terms in \eqref{eq: Liouville measures new Poincare inequality proposition}. However, the constant $1/(\abs{\alpha}-1)$ 
	in \eqref{eq: Liouville measures Poincare inequality with s^2 corollary} 
	is smaller than all the constants $1/(\alpha_i-1)$ in \eqref{eq: Liouville measures new Poincare inequality proposition}.
	
	As announced, we present examples of weighted Poincar\'e inequalities for multivariate Liouville distributions. Specifically, we apply Theorem \ref{theorem: Liouville measures Poincare inequality} to obtain inequalities for the Gamma Liouville and the inverted Dirichlet distributions, both introduced at the beginning of this subsection. 
	In addition, we establish another weighted Poincar\'e inequality for the inverted Dirichlet distribution, obtained directly through a direct transport argument involving the Dirichlet distribution.
	\begin{example}
		Consider the Gamma Liouville distribution $\mu$ with density given in \eqref{eq: Multivariate Liouville example correlated Gammas}. Note that the application of Theorem \ref{theorem: Liouville measures Poincare inequality} reduces to find an available weight $w_S\in \weightClass{\Rplus}$ for the associated one-dimensional probability measure $\mu_S$ appearing in the theorem. In the present setting this measure is Gamma distributed, with density 
		\[s \in \Rplus\mapsto \frac{1}{\Gamma\left(\abs{\alpha}+\delta\right)}s^{\delta+\abs{\alpha}-1}e^{-s}.\]
		Then a natural choice of weight is the Stein kernel of $\mu_S$, given by $w_s(s)=s$, and for which the optimal Poincar\'e constant is $C_P(\mu_S,w_S)=1$ (see \eg \cite{stein_kernels_ley}). For this weight the constant in inequality \eqref{eq: Liouville measures Poincare inequality theorem} becomes
		\[\max\left(C_P(\mu_S,w_S),\frac{K}{\abs{\alpha}}\right)=\max\left(1,\frac{1}{\abs{\alpha}}\int_{\Rplus}s\,d\mu_S(s)\right)
		=\max\left(1,\frac{1}{\abs{\alpha}}(\abs{\alpha}+\delta)\right)=1+\frac{\delta}{\abs{\alpha}}.\]
		Consequently, Theorem \ref{theorem: Liouville measures Poincare inequality} entails the following weighted Poincar\'e inequality for $\mu$:
		\begin{equation*}
			\Var_\mu(f)
			\leq \left(1+\frac{\delta}{\abs{\alpha}}\right)
			\int_{\Rplus^n}
			\sum_{i=1}^{n} x_i\, \left(\frac{\partial f}{\partial x_i}(x)\right)^2\,
			d\mu.
		\end{equation*}
		This inequality is optimal when $\delta=0$, that is, when $\mu$ is a product of independent gamma distributions. In this case it coincides with the well-known weighted Poincar\'e inequality obtained by tensorizing one-dimensional Poincar\'e inequalities for gamma distributions, using their Stein kernels as weights (see \eg \cite{miclo_dirichlet}).
		\\
		Other potential choices of weight include the constant weight $w_s(s)=1$ or the quadratic one $w_s(s)=s^2$.
	\end{example}
	
	\begin{example}
		\label{example: inverted Dirichlet}
		Let us turn our attention to the inverted Dirichlet distribution $\dirichInv$ defined in \eqref{eq: Multivariate Liouville example inverted Dirichlet distribution}. It is a heavy-tailed distribution, meaning that it does not admit exponential moments. Hence it does not satisfy a classical Poincar\'e inequality (see \cite{BGL}). However it satisfies weighted Poincar\'e inequalities, as we present now. 
		
		The associated one-dimensional probability measure $\mu_S$ is a Beta prime distribution, with density proportional to
		\[s\in \Rplus \mapsto 
		\,s^{\abs{\alpha}-1}\,(1+s)^{-(\abs{\alpha}+\alpha_0)}.\]
		Using the transport argument \eqref{eq: weighted Poincare inequality by transport} we show that this measure satisfies a weighted Poincar\'e inequality with weight $w_S(s)=s^2$. Indeed, one can easily prove that $\mu_S=T_\# \Pi_\beta$, where $\Pi_{\beta}$ is the beta distribution with vector parameter $(\beta_1,\beta_2)=(\abs{\alpha},\alpha_0)$ and $T$ is the mapping given by $T(t)=t/(1-t)$ for all $t\in (0,1)$.  As proved in the Appendix, the measure $\Pi_\beta$ satisfies a weighted Poincar\'e inequality with weight $w_\tau(t)=t^2(1-t)^2$. The optimal constant $C_P(\Pi_\beta,w_\tau)$ is bounded from below by $4\max\left(\frac{1}{\beta_1^2},\frac{1}{\beta_2^2} \right)$ and bounded from above by the expression
		\begin{equation*}
			\Phi(\beta_1, \beta_2) = \begin{cases} 
				4\max \left( \dfrac{1}{\beta_1^2}, \dfrac{1}{\beta_2^2} \right), & \text{if } \min(\beta_1, \beta_2) \leq 1, \\[12pt]
				4\max \left( \dfrac{1}{2\beta_1 - 1}, \dfrac{1}{2\beta_2 - 1} \right), & \text{if } \min(\beta_1, \beta_2) \geq 1.
			\end{cases}
		\end{equation*}
		Note that in the regime $\min(\beta_1,\beta_2)\leq 1$ the bounds are sharp, leading to the identity  $C_P(\Pi_\beta,w_\tau)=4\max\left(\frac{1}{\beta_1^2},\frac{1}{\beta_2^2}\right)$. The upper bound is also known to be tight in the symmetric case $\beta_1=\beta_2=\kappa\geq 1$, where we have 
		$C_P(\Pi_\beta,w_\tau)=4/(2\kappa-1)
		$ (see \cite{HerediaWeightPoincare}). Finally, since for all $t\in (0,1)$ we have
		\[w_\tau(t)(T'(t))^2=t^2(1-t)^2 \frac{1}{(1-t)^4}=\left(\frac{t}{1-t}\right)^2=(T(t))^2,\]
		then the transport argument emphasized in \eqref{eq: weighted Poincare inequality by transport} entails that $\mu_S$ satisfies a weighted Poincar\'e inequality with weight $w_S(s)=\left(w_\tau\, (T')^2\right)\circ T^{-1}(s)=s^2$. Moreover, since the optimal Poincar\'e constant is invariant under transport,
		we have that $C_P(\mu_S,w_S)\leq \Phi(\abs{\alpha},\alpha_0)$. 
		
		We can now return to the multivariate Dirichlet distribution $\dirichInv$.
		After applying Theorem \ref{theorem: Liouville measures Poincare inequality} with the weight $w_S$ we obtain the weighted Poincar\'e inequality below
		\begin{equation}
			\label{eq: Poincare inequality inverted Dirichlet not used}
			\Var_{\dirichInv }(f)
			\leq \max\left(\frac{1}{\abs{\alpha}},\Phi(\abs{\alpha},\alpha_0)\right)\,
			\int_{\Rplus^{n}}s(x) 
			\sum_{i=1}^{n}\,x_i \left(\frac{\partial f}{\partial x_i}(x)\right)^2\,
			d\dirichInv.
		\end{equation}
		
		We propose another weighted Poincar\'e inequality for $\dirichInv$ obtained directly using transport. Consider the Dirichlet distribution $\dirichTilde$ in $\dirichDomTilde$ with vector parameter $\Tilde{\alpha}=(\alpha_1,\dots,\alpha_n,\alpha_0)$. Then we have $\dirichInv=T_\# \dirichTilde$, where $T$ is the mapping defined as
		\[T(t)=\frac{t}{1-\sum_{i=1}^{n} t_{i}},\quad t=(t_1,\dots,t_n)\in \dirichDomTilde,\]
		(see \cite{dirichlet_book}). Recall that $\uno\in \R^n$ denotes the column vector with all entries equal to one. The inverse mapping and the Jacobian of $T$ are explicitly given by
		\[T^{-1}(x)=\frac{x}{1+s(x)},\quad x\in\Rplus^{n},\]
		and
		\[\Jac(T)(t)=\frac{1}{1-\sum_{i=1}^{n} t_{i}}\left(\Id+\,\frac{1}{1-\sum_{i=1}^{n} t_{i}}t\, \uno^\top\right),\quad t\in \dirichDomTilde,\]
		respectively. 
		\\
		Recall that $\dirichTilde$ satisfies a weighted Poincar\'e inequality with weight
		\[\Weight_\tau(t)=\diag(t)-t\, t^\top,\quad t\in \dirichDomTilde,\]
		as claimed in \eqref{eq: Poincare inequality Dirichlet simplex}. Observe that for any $x\in \Rplus^n$ we have
		\[
		\Jac(T)\circ\, T^{-1}(x)=(s(x)+1)\left(\Id+x\,\uno^\top\right),\quad  W_\tau\circ\,T^{-1}(x)=\frac{1}{1+s(x)}\diag(x)-\frac{1}{(1+s(x))^2}x \,x^\top.\]
		Applying then the transport argument in \eqref{eq: weighted Poincare inequality by transport}, it follows that $\dirichInv$ satisfies a Poincar\'e inequality with weight $W$ defined for all $x\in \Rplus^n$ as
		\begin{align*}
			W(x)=\left(\Jac(T)\Weight_\tau \Jac(T)^\top\right)\circ T^{-1}(x)&=(\Id+x \uno^\top)\left( (1+s(x))\diag(x)-x x^{\top} \right)\left(\Id+\uno x^\top\right),
		\end{align*}
		and optimal constant $C_P(\dirichInv,\Weight)=1/(\abs{\alpha}+\alpha_0)$.
		After some algebraic simplifications involving the identities $ \diag(x)\,\uno =x$ and $\uno^\top x=s(x)$ we are able to reduce this weight to
		\begin{align*}
			W(x)=(1+s(x))\left(\diag(x)+x\,x^\top\right).
		\end{align*}
		In other words, the measure $\dirichInv$ satisfies the weighted Poincar\'e inequality
		\begin{equation}
			\label{eq: Poincare inequality inverted Dirichlet}
			\Var_{\dirichInv }(f)
			\leq \frac{1}{\abs{\alpha}+\alpha_0}\,
			\int_{\Rplus^{n}} (1+s(x)) \left(\sum_{i=1}^n x_i \left(\frac{\partial f}{\partial x_i}(x)\right)^2+\left(\sum_{i=1}^n x_i \frac{\partial f}{\partial x_i}(x)\right)^2\right)\,
			d\dirichInv.
		\end{equation}
		This inequality will later be combined with a copula-based argument to establish weighted Poincar\'e inequalities for distributions whose structure of dependence is encoded by the Clayton copula. Such a methodology which will be applied in global sensitivity analysis in Section \ref{section: Application to Global Sensitivity Analysis}. Although inequality \eqref{eq: Poincare inequality inverted Dirichlet not used} could also be employed for this purpose, we instead rely on the one with a known optimal constant, as this is crucial to ensure accuracy in the numerical experiments.
		
	\end{example}
	\section{Digression: weighted Poincar\'e inequalities for elliptically contoured distributions}
	\label{sec: Weighted Poincare inequalities for elliptically contoured distributions}
	Although elliptical contoured distributions are not a central focus of the present paper, they are considered in this section due to their importance in multivariate statistics and its applications. For instance, they play a central role in financial mathematics \cite{elliptical_distributions_in_statistics}, robust statistics \cite{symmetric_multivariate_distributions,robust_statistics} and high-dimensional covariance matrix estimation
	\cite{high-dimensional_matrix_estimation}. They also provide a natural framework for modeling multivariate dependence in terms of copulas, as discussed in the next section. An overview on these distributions can be found, for instance, in \cite{symmetric_multivariate_distributions}. 
	
	To preserve our formalism based on continuous probability measures, we consider elliptical contoured distributions having a density. However, the results presented in this part remain valid for more general elliptical distributions, defined in terms of stochastic representations (see \cite{symmetric_multivariate_distributions}). 
	
	We say that a probability measure $\mu\in \probClass{\R^n}$ is elliptically contoured (or simply an elliptical distribution) if its density function is of the form
	\begin{align*}
		\rho(x)= 
		Z^{-1}
		g\left( \norm{x-m}_{\Sigma^{-1}}\right),\quad x\in \R^n,
	\end{align*}
	where the vector $m\in \R^n$ and the positive definite matrix $\Sigma\in \matClass^n$ are fixed parameters and $g\colon \Rplus\rightarrow \R$ is an univariate positive function.
	For simplicity we assume $m=0$ and, following standard convention, we suppose without loss of generality that $\Sigma$ belongs to the subclass of matrices
	\begin{equation}
		\label{eq: matClassR matrices renormalized}
		\matClassR^n=\left\{ \Sigma\in \matClass^n \br \Sigma_{i,i}=1,\, \Sigma_{i,j}\in (-1,1), \mbox{for }i\neq j,\, i,j\in \nset{1}{n} \right\}.
	\end{equation}
	The density $\rho$ offers a clear geometric characterization of these measures, as its level sets consist of unions of hyperellipsoids. In the particular case $\Sigma=\Id$, these level sets are unions of spheres, and then we say that $\mu$ is a spherical contoured distribution (or simply a spherical distribution). To emphasize the difference between elliptical and spherical distributions, we denote them by $\mu_\Sigma$ and $\mu_{\Id}$, respectively.
	
	Recall some important examples of both elliptical and spherical distributions. 
	\begin{enumerate}
		\item A fundamental elliptical distribution is the multivariate normal $\mu\sim \normal(0,\Sigma)$, obtained after taking the function $g(r) =
		\exp\left(-r^2/2\right)$, $r\in \Rplus$. Recall that its density function is given by
		\[\rho(x)=(2\pi)^{-\frac{n}{2}} \det(\Sigma)^{-\frac{1}{2}} e^{-\frac{1}{2} \norm{x}_{\Sigma^{-1}}^2},\quad  x\in \R^n.\]
		The spherical version of this measure is the standard multivariate normal distribution $\mu_{\Id}\sim \normal(0,\Id)$.
		\item Another well-known example of elliptical distribution is the multivariate t-distribution, which generalizes the Student's t-distribution. Its density is defined as
		\begin{equation}
			\label{eq: multivariate t-distribution}
			\rho(x)=
			Z^{-1}
			\left(1+\frac{1}{\gamma}\norm{x}^2_{\Sigma^{-1}}\right)^\mlap{-\frac{\gamma+n}{2}},\qquad x\in \R^n,
		\end{equation}
		with $\gamma>0$, that we recover when we take $g(r)$
		proportional to
		$(1+r^2/\gamma)^{-\frac{\gamma+n}{2}}$. We denote this measure by $\mu_{\Sigma,\gamma}$.
		\\
		The corresponding spherical distribution $\mu_{\Id,\gamma}$ is nothing but an affine transformation of the generalized Cauchy distribution. The latter, denoted $\nu_\kappa$, depends on a parameter $\kappa>n/2$ and has the density
		\begin{equation*}
			x\in \R^n\mapsto 
			Z^{-1}
			(1+\norm{x}^2)^{-\kappa}.
		\end{equation*}
		\item Finally, consider the symmetric multivariate logistic distribution, that we have taken from \cite{elliptical_distributions_in_statistics}, with density function
		\begin{equation}
			\label{eq: multivariate logistic density}
			\rho(x)=
			Z^{-1}
			\frac{e^{-\norm{x}^2_{\Sigma^{-1}}} } {\left(1+ e^{-\norm{x}^2_{\Sigma^{-1}}}\right)^2}, \quad x\in \R^n.
		\end{equation}
		This measure is obtained by choosing $g(r)=
		e^{-r^2} (1+e^{-r^2})^{-2}$.
	\end{enumerate}
	
	Any elliptical distribution is related to its spherical counterpart $\mu_\Sigma$ through a simple linear transformation. More precisely, $\mu_\Sigma$ is the image measure of $\mu_\Id$ under
	the linear mapping $x\mapsto T(x)=\Sigma^{\frac{1}{2}} x$, where $\Sigma^{\frac{1}{2}}\in \matClass^n$ is the unique positive definite square root of $\Sigma$. One can then apply the transport argument in \eqref{eq: weighted Poincare inequality by transport} to obtain
	weighted Poincar\'e inequalities for any elliptical measure, once the spherical case has been treated.
	\\
	In the spherical setting, inequalities 
	have already been investigated. See for instance Theorem 5.1 in \cite{joulin_bonnefont_spherically_symmetric}, which can be seen as an analogue of Theorem \ref{theorem: Liouville measures Poincare inequality} for spherical distributions $\mu_\Id$. When applying this result, the resulting weights take the natural form $ \Weight_\Id(x)=w(\norm{x})\,\Id$, where $w$ is a one-dimensional function. This type of weights leads to simple inequalities in the elliptical setting, as stated in the following proposition. Its proof follows directly from the transport argument in \eqref{eq: weighted Poincare inequality by transport}.
	\begin{prop}
		\label{prop: elliptical measures Poincare inequality}
		Let $\mu_\Sigma$ be an elliptical distribution with matrix parameter $\Sigma\in \matClassR^n$. Suppose that the corresponding spherical distribution $\mu_\Id$ satisfies a weighted Poincar\'e inequality with weight $W_\Id(x)=w(\norm{x})\,\Id$, where $w\in \weightClass{\Rplus}$. Then $\mu_\Sigma$ satisfies the inequality
		\begin{equation*}
			\Var_{\mu_\Sigma}(f)\leq C_P(\mu_{\Id},\Weight_\Id)\int_{\R^n} w\left(\norm{x}_{\Sigma^{-1}}\right) \norm{\nabla f}_{\Sigma}^2\,d\mu_{\Sigma},
		\end{equation*}
		for which the constant $C_P(\mu_{\Id},\Weight_\Id)$ is optimal.
	\end{prop}
	
	We illustrate this inequality for the examples introduced above.
	
	\begin{example}
		Consider the multivariate normal distribution $\mu_\Sigma\sim \normal(0,\Sigma)$. Applying Proposition \ref{prop: elliptical measures Poincare inequality} to this measure we obtain a result which coincides with a classical one. However, it is interesting to observe that it can be recovered directly from the standard spherical case where $\mu_\Id\sim \normal(0,\Id)$. 
		
		It is well-known that the measure $\mu_\Id$ satisfies a classical Poincar\'e inequality (\ie with weight $\Weight_\Id\equiv\Id$) and optimal constant $C_P(\mu_\Id,\Id)=1$. Thus, from Proposition \ref{prop: elliptical measures Poincare inequality} it follows that $\mu_\Sigma$ satisfies the weighted inequality
		\begin{equation}
			\label{eq: weighted Poincare inequality normal gaussian}
			\Var_{\mu_\Sigma}(f)\leq \int_{\R^n} \norm{\nabla f}_{\Sigma}^2\,d\mu_{\Sigma}.
		\end{equation}
		The same result can also be directly obtained applying the Brascamp-Lieb inequality \eqref{eq: Brascamp-Lieb}.
	\end{example}
	\begin{example}
		\label{example: Example t-student}
		Let now $\mu_{\Sigma,\gamma}$ denote the multivariate t-distribution with density given in \eqref{eq: multivariate t-distribution}. Since this measure is heavy-tailed, it does not satisfy a classical Poincar\'e inequality. However, it satisfies a weighted Poincar\'e inequality obtained from the spherical case. Indeed, the spherical distribution $\mu_\Id$ is the image measure of the generalized Cauchy distribution $\nu_{(\gamma+n)/2}$ under the mapping $ T(x)=\gamma^{\frac{1}{2}} x$. An available weight for the latter measure is the function $x\in \R^n\mapsto (1+\norm{x}^2)\,\Id$. Due to transport, it induces a weighted Poincar\'e inequality for $\mu_{\Id,\gamma}$, with weight $ \Weight_{\Id,\gamma}(x)=(\gamma+\norm{x}^2)\,\Id$. Moreover, it preserves the same optimal Poincar\'e constant, which has been completely determined in \cite{huguet_cauchy}, showing that
		\begin{equation}
			\label{eq: optimal constant Cauchy}
			C_P(\mu_{\Id,\gamma},\Weight_{\Id,\gamma})=C_\gamma:=
			\begin{cases}
				\dfrac{4}{\gamma^2}, & \mbox{if } 0<\gamma\leq 4, \\
				\dfrac{1}{2\gamma-4}, & \mbox{if }4\leq \gamma \leq n+2, \\
				\dfrac{1}{\gamma+n-2}, & \mbox{if } n+2\leq \gamma.
			\end{cases}
		\end{equation}
		Based on this inequality we obtain a corresponding result for the elliptical measure
		$\mu_{\Sigma,\gamma}$. From from Proposition \ref{prop: elliptical measures Poincare inequality} it follows that
		\begin{equation*}
			\Var_{{\mu}_{\Sigma,\gamma}}(f)\leq C_\gamma \int_{\R^n} (\gamma+\norm{x}_{\Sigma^{-1}}^2)\norm{\nabla f}_{\Sigma}^2\,d\mu_{\Sigma,\gamma}.
		\end{equation*}
	\end{example}
	\begin{example}
		Finally, we turn our attention to the symmetric multivariate logistic distribution $\mu_\Sigma$ with density in \eqref{eq: multivariate logistic density}. We first establish a weighted Poincar\'e inequality for the spherical measure $\mu_\Id$ using its radial part, which is represented by the one-dimensional probability measure $\muR\in\probClass{\Rplus}$ with density proportional to
		\[r\in \Rplus\mapsto 
		r^{n-1}\, 
		e^{-r^2} (1+e^{-r^2})^{-2}.\]
		In the Appendix we prove that $\muR$ satisfies a weighted Poincar\'e inequality with weight $\wR(r)=r^2$ and optimal constant $C_P(\muR,\wR)$ which is bounded from above by
		\[D_n =\begin{cases}
			\dfrac{4}{n^2}, & \mbox{if }n\leq 4,\\[12pt]
			\dfrac{1}{2(n-2)}, & \mbox{if }n\geq 4.
		\end{cases}\]
		Therefore, applying Theorem 5.1 in \cite{joulin_bonnefont_spherically_symmetric} it follows that $\mu_\Id$ satisfies a weighted Poincar\'e inequality with weight $ \Weight_\Id(x)=\norm{x}^2 \Id$ and optimal constant satisfying the bound
		\[C_P(\mu_\Id,\Weight_\Id)\leq \max\left(C_P(\muR,\wR),\frac{1}{n-1}\int_{\Rplus} \frac{r^2}{\wR(r)} \,d\muR(r) \right)= \frac{1}{n-1}.\]
		Therefore by Proposition \ref{prop: elliptical measures Poincare inequality} the elliptical measure $\mu_\Sigma$ satisfies
		\begin{equation*}
			\Var_{\mu_\Sigma}(f)\leq \frac{1}{n-1} \int_{\R^n} \norm{x}_{\Sigma^{-1}}^2\norm{\nabla f}_{\Sigma}^2\,d\mu_{\Sigma}. 
		\end{equation*}
		In the case $n=1$, $\mu_\Sigma$ matches with the radial measure $\muR$ and we have $C_P(\muR,\wR)\leq 4$.
	\end{example}
	
	\section{Weighted Poincar\'e inequalities for measures with prescribed copulas}
	\label{sec: Measures with prescribed copulas}
	In this part, rather than focusing on a specific class of probability distributions, we discuss a general approach to establish weighted Poincar\'e inequalities for measures whose structure of dependence is determined by a known copula. This situation is particularly important in statistical applications, where probability distributions
	are often modeled according to classical copulas. For comprehensive introductions to copulas, we refer to \cite{Nelsen_copulas,copulas_harry}.

	Let $\Omega\subset \R^n$ be a connected open set with a piecewise $\C^1$ boundary. Let $\mu \in \probClass{\Omega}$ be a probability measure with density $\rho$. Denote the cumulative distribution function (cdf) of $\mu$ by
	\[
	F(x_1,\dots,x_n)=\mu\!\left(\Omega \cap\prod_{i=1}^n (-\infty,x_i)\right),
	\quad x\in \R^n.
	\]
	For each $i\in \{1,\dots,n\}$, let $\mu_i$ denote its $i$-th marginal distribution, with density $\rho_i$ and cdf $F_i$.
	By Sklar's theorem (Theorem 2.3.3 in \cite{Nelsen_copulas}), there exists a unique function
	\[(u_1,\dots,u_n)\in [0,1]^n\mapsto \cop(u_1,\dots,u_n)\in [0,1],\]
	called the copula of $\mu$, such that
	\begin{equation}
		\label{eq: copula definition}
		F(x_1,\dots,x_n)=\cop\big(F_1(x_1),\dots,F_n(x_n)\big),
		\quad (x_1,\dots,x_n)\in\R^n.
	\end{equation}
	The copula $\cop$ encodes the structure of dependence of $\mu$, so that the measure is completely determined by its marginals and $\cop$. For instance, the density can be decomposed as
	\begin{equation}
		\label{eq: density from copula}
		\rho(x)=\prod_{i=1}^n\rho_i(x_i)\,\frac{\partial^n\, \cop}{\partial u_1,\dots, \partial u_n} \left(F_1(x_1),\dots,F_n(x_n)\right),\quad x\in \Omega,
	\end{equation}
	which follows immediately from \eqref{eq: copula definition}. Various aspects of dependence can be measured in terms of $\cop$, such as concordance and tail dependence  (see \cite{Nelsen_copulas,copulas_harry}). 
	
	Encoding in such a way the structure of dependence simplifies the transport between distributions sharing a common copula. More precisely, if two distributions $\mu\in \probClass{\Omega}$ and $\mutild\in\mathcal{P}(\Tilde{\Omega})$ have the same copula, then the transport between them is reduced to the transport of their marginals.
	In other words, $\mu$ is the image measure of $\mutild$ by the Nataf transformation, defined as
	\begin{equation}
		\label{eq: Nataf transformation}
		T(\tildx)
		=(T_{1}(\tildx_1),\dots,T_{n}(\tildx_n))
		=(\FX{1}^{-1}\circ \FtilX{1}(\tildx_1),\dots,\FX{n}^{-1}\circ \FtilX{n}(\tildx_n)),\quad \tildx\in \Tilde{\Omega},
	\end{equation}
	where each $\FtilX{i}$ is the cdf of the marginal $\mutild_i$, of density $\rhotild_i$.
	Recall that in this one-dimensional setting, each $T_i$ is the optimal transport map with respect to the $p$-Wasserstein distance, for all
	$1\leq p<+\infty$. In the present framework where $\mu$ and $\mutild$ share the same copula, this property extends to the Nataf transform $T$ (see \cite{nataf_optimal}). The inverse mapping of $T$ is also a Nataf transformation, from $\mu$ to $\mutild$, given by
	\begin{equation*}
		T^{-1}(x) = (T_1^{-1}(x_1),\dots,T_n^{-1}(x_n)) = (\FtilX{1}^{-1}\circ \FX{1}(x_1), \dots, \FtilX{n}^{-1}\circ \FX{n}(x_n)), \quad  x \in \Omega.
	\end{equation*}
	
	Using the relation $\mu=T_{\#} \mutild$, weighted Poincar\'e inequalities for $\mu$ can be obtained directly from those satisfied by $\mutild$ via the transport argument in \eqref{eq: weighted Poincare inequality by transport}. This is formalized in the following proposition. In this case, since the map $T$ acts component-wise, the matrix $D=\Jac(T)\circ T^{-1}$ is diagonal.
	\begin{prop}
		\label{prop: copula Poincare inequality}
		Let $\mu$ and $\mutild$ be two probability distributions sharing the same copula. Let $T$ be the Nataf transformation defined in \eqref{eq: Nataf transformation}.
		Suppose that $\mutild$ satisfies a weighted Poincar\'e inequality with weight $\Tilde{\Weight}$. Then the probability measure $\mu$ satisfies the following inequality
		\begin{equation}
			\label{eq: copula weighted Poincare inequality}
			\Var_\mu(f)\leq C_P(\mutild,\wtild)\int_{\Omega} \norm{D\, \nabla f}^2_{\Tilde{W}\circ T^{-1}}\,d\mu,
		\end{equation}
		where $D$ is the diagonal matrix with entries
		\begin{equation}
			\label{eq: diagonal term D}
			D_{i,i}=T_i'\circ T_i^{-1}=\frac{\rhotild_i}{\rho_i\circ
				F_{i}^{-1}\circ \FtilX{i}}\circ (T_i^{-1})  =\frac{\rhotild_i\circ T_i^{-1}}{\rho_i}.
		\end{equation}
	\end{prop}
	Note that this inequality remains valid in the one-dimensional case $n=1$. Actually, in this setting all distributions share the unique copula $C(u)=u$, $u\in [0,1]$. Consequently, for any one-dimensional measures $\mu$ and $\mutild$ and any available weight $\Tilde{w}$, we obtain
	\begin{equation}
		\label{eq: copula weighted Poincare inequality 1D}
		\Var_\mu(f)\leq C_P(\mutild,\Tilde{w})\int_{\Omega} \left(\tilde{w}\circ T^{-1}(x)\right) \,\frac{\rhotild \circ T^{-1}(x)}{\rho(x)}\,\left(f'(x)\right)^2\,d\mu(x).
	\end{equation}
	
	In practice, to establish a weighted Poincar\'e inequality for a given probability measure $\mu$ with copula $\cop$  --for instance, a measure arising from a statistical model-- one first selects a reference measure $\mutild$ sharing the same copula and for which an inequality is known. Then the corresponding inequality for $\mu$ is given in \eqref{eq: copula weighted Poincare inequality}. Whenever the reference densities $\rhotild_i$ and quantile functions $\FtilX{i}^{-1}$ admit closed-form expressions, it can be written explicitly in terms of $\rho_i$ and $F_{i}$. Below we illustrate this approach with examples involving classical copulas. 
	
	\mathTitle{The Clayton copula}
	Consider the $n$-variate Clayton copula $\cop_\theta$ of parameter $\theta>0$, defined as
	\[\cop_\theta(u)=\left(1+\sum_{i=1}^n (u_i^{-\theta}-1)\right)^\mlap{-\frac{1}{\theta}},\qquad u\in [0,1]^n.\]
	This copula is commonly used to model lower tail dependence, which is completely determined by the parameter $\theta$ (see \eg \cite{copulas_harry}).
	This parameter also determines the concordance between marginals, measured in terms of Kendall rank correlation coefficient \cite{tau_elliptical_copulas}.
	
	As a reference measure with copula $\cop_\theta$, consider $\mutild$ defined in the strictly negative orthant $\Rminus^n=(-\infty,0)^n$ with density given by
	\begin{equation*}
		\rhotild(\Tilde{x})= \frac{\Gamma\left(\frac{1}{\theta}+n\right)}{\Gamma\left(\frac{1}{\theta}\right)} \left(1-\sum_{i=1}^n \tildx_i\right)^\mlap{-\left(\frac{1}{\theta}+n\right)},\qquad \quad \tildx\in \Rminus^n.   
	\end{equation*}
	It corresponds to the image measure of the multivariate Pareto distribution in \eqref{eq: multivariate pareto} with parameter $\gamma=1/\theta$, under the mapping $x\in \Rplus\mapsto -x$. We verify that $\cop_\theta$ is the copula of
	$\Tilde{\mu}$. Indeed, the marginal distributions of $\mutild$ have common densities and cdf given by
	\[\rhotild_i(\tildx_i)=\frac{1}{\theta}(1-\tildx_i)^{-\left(\frac{1}{\theta}+1\right)},\qquad \FtilX{i}(\tildx_i)=(1-\tildx_i)^{-\frac{1}{\theta}},\qquad \mbox{for all }\tildx_i\in {\Rminus}\,.\]
	Replacing these functions and the copula $\cop_\theta$ in \eqref{eq: density from copula}, we recover the density $\rhotild$:
	for all $\tildx\in \Rminus^n$ we have
	\begin{multline*}
		\prod_{i=1}^n \tilde{\rho}_i(\tilde{x}_i) \frac{\partial^n \mathcal{C}_\theta}{\partial u_1 \dots \partial u_n} \left( \tilde{F}_1(\tilde{x}_1), \dots, \tilde{F}_n(\tilde{x}_n) \right) \\
		= \prod_{i=1}^n \left( \frac{1}{\theta} (1-\tilde{x}_i)^{-\left(\frac{1}{\theta}+1\right)} \right) \left[ \prod_{i=1}^n \left( \theta\, (1-\tilde{x}_i)^{ \frac{1}{\theta}+1} \right) \prod_{k=0}^{n-1} \left( \frac{1}{\theta} + k \right) \left( 1 - \sum_{i=1}^n \tilde{x}_i \right)^{-\left(\frac{1}{\theta} + n\right)}\right]= \rhotild(\tildx),
	\end{multline*}
	meaning that $\cop_\theta$ is the copula of $\Tilde{\mu}$ by uniqueness.
	
	In this way, given any other probability measure $\mu\in \probClass{\Omega}$ having the Clayton copula $\cop_\theta$,
	we can use a weighted Poincar\'e inequality for $\mutild$ as a proxy to obtain a corresponding inequality for $\mu$. For instance, recall that in \eqref{eq: Poincare inequality inverted Dirichlet} we provide an inequality with explicit optimal constant for the inverted Dirichlet distribution $\dirichInv$, which generalizes the multivariate Pareto distribution. Choosing the parameters $\alpha=\uno$, $\alpha_0=1/\theta$ and performing a change of sign we obtain the following weight for $\mutild$:
	\[\wtild(\tildx)=\left(1+\sum_{i=1}^n (-\tildx_i)\right)(\diag(-\tildx)+(-\tildx) (-\tildx)^\top),\qquad \tildx\in \Rminus^n,\]
	with optimal constant $C_P(\mutild,\wtild)=\theta/(1+n\,\theta)$. Now, since we can write explicitly
	\[\FtilX{i}^{-1}(u_i)=1-u_i^{-\theta},\qquad \rhotild_i\circ \FtilX{i}^{-1} (u_i)=\frac{1}{\theta}u_i^{\theta+1},\qquad u_i\in (0,1),\] 
	then the corresponding inequality 
	for $\mu$ given in Proposition \ref{prop: copula Poincare inequality} takes the form
	\begin{equation}
		\label{eq: weighted Poincare inequality clayton copula}
		\scaleobj{0.95}{
			\dsp
			\Var_\mu(f)\leq \frac{\theta}{1+n\,\theta} \int_\Omega \left(1+\sum_{i=1}^n T_i^{-1}\right)\left[\sum_{i=1}^n T_i^{-1}\left(D_{i,i}\,\frac{\partial f}{\partial x_i}\right)^2+\left(\sum_{i=1}^n T_i^{-1}\,D_{i,i}\,\frac{\partial f}{\partial x_i}\right)^2\right] \,d\mu,}
	\end{equation}
	where
	\[T_i^{-1}(x_i)=F_i(x_i)^{-\theta}-1\qquad\mbox{and}\qquad D_{i,i}(x_i)=\frac{1}{\theta}\frac{F_i(x_i)^{\theta+1}}{\rho_i(x_i)}.\]
	Despite its apparent complexity, this explicit inequality finds relevant applications in global sensitivity analysis, as presented in the forthcoming Section \ref{section: Application to Global Sensitivity Analysis}.
	
	\mathTitle{Elliptical copulas}
	In Section \ref{sec: Weighted Poincare inequalities for elliptically contoured distributions} we considered the case of elliptical distributions $\mu_\Sigma$, where $\Sigma$ belongs to the class of matrices $\matClassR^n$ in \eqref{eq: matClassR matrices renormalized}. The copulas $\cop_\Sigma$ arising from distributions of this type are also called elliptical.  They provide a flexible framework for modeling dependence. Among other reasons, this is because if a probability measure $\mu$
	has an elliptical copula, say $\cop_\Sigma$, then the concordance between its marginals is entirely determined by the matrix parameter $\Sigma$ (see \eg \cite{elliptical_distributions_dependence}). Consequently, the level of concordance can be adjusted conveniently in situations where the practitioner is free to choose $\Sigma$. 
	
	While a distribution $\mu$ may possess an elliptical copula $\cop_\Sigma$, it is not necessarily elliptical itself. General distributions having elliptical copulas are referred to as meta-elliptical (see \cite{meta_elliptical_distributions}). Our approach thus yields weighted Poincar\'e inequalities for such distributions by taking the reference measure $\mutild$ to be the elliptical one $\mu_\Sigma$ associated with the copula $\cop_\Sigma$, whenever an inequality for this measure is available.
	\\
	For instance, consider the multivariate t-distribution $\mutild=\mu_{\Sigma,\gamma}$ defined in \eqref{eq: multivariate t-distribution}, with $\gamma>0$. Its associated copula $\mathcal{C}_{\Sigma,\gamma}$ is known as the t-copula. As shown in Example \ref{example: Example t-student}, an available weight for this measure is $\wtild(x)=\left(\gamma+\norm{x}^2_{{\Sigma}^{-1}}\right)\Sigma$. Then, if $\mu\in \probClass{\Omega}$ is a meta-elliptical measure with copula $\mathcal{C}_{\Sigma,\gamma}$, its corresponding weighted Poincar\'e inequality
	given by Proposition \ref{prop: copula Poincare inequality} is
	\[\Var_\mu(f)\leq C_\gamma \int_{\Omega}(\gamma+\norm{T^{-1}(x)}^2_{\Sigma^{-1}})\,\norm{D(x)\,\nabla f (x)}_\Sigma^2 \,d\mu(x),\]
	where the optimal constant $C_\gamma$ is given in \eqref{eq: optimal constant Cauchy}. 
	\\
	When $\gamma\in \{1,2\}$, the functions $\rhotild_i$ and $\FtilX{i}^{-1}$ admit closed-form expressions, allowing $T^{-1}$ and $D$ to be expressed in terms of the target marginal functions $\rho_i$ and $F_i$. Indeed, the marginals of $\mutild$ are Student's t-distributions with a common density given by
	\[\rhotild_i(\tildx_i)=\frac{\Gamma\left(\frac{\gamma+1}{2}\right)}{(\gamma\pi)^{\frac{1}{2}}\Gamma\left(\frac{\gamma}{2}\right)}\left(1+\frac{1}{\gamma}\tildx_i^2\right)^\mlap{-\frac{\gamma+1}{2}},\qquad \tildx_i\in \R.\] 
	When $\gamma=1$, for instance, each marginal $\mutild_i$ reduces to the classical Cauchy distribution, for which
	\[\rhotild_i(\tildx_i)=\frac{1}{\pi(1+\tildx_i^2)}, \quad \FtilX{i}^{-1}(u_i)=\tan\left(\pi\left(u_i-\frac{1}{2}\right)\right), \quad  \tildx_i \in \mathbb{R},\; u_i \in (0,1).\]
	Another classical example of elliptical copula is the Gaussian one, associated with the multivariate normal distribution. We treat this case separately, providing a general result.
	
	\mathTitle{The Gaussian copula}
	The Gaussian copula $\cop_\Sigma$ is the elliptical copula associated with the multivariate normal distribution $\mu_\Sigma\sim\normal(0,\Sigma)$, where $\Sigma\in \matClassR^n$ by convention. There is no closed-form expression for $\cop_\Sigma$. However, from \eqref{eq: density from copula} one can easily see that the density of any probability measure $\mu\in\probClass{\Omega}$ with this copula takes the form
	\begin{equation}
		\label{eq: gaussian copula density}
		\dsp\rho(x) = \prod_{i=1}^n \rho_i(x_i) \frac{1}{\sqrt{\det (\Sigma)}} \exp \left( -\frac{1}{2} \norm{\bigl(\Phi^{-1}\circ F_1(x_1),\dots,\Phi^{-1}\circ F_n(x_n)\bigr)}_{\,\left(\Sigma^{-1}-\scaleobj{0.9}{\Id}\right)}^2 \right),
	\end{equation}
	for all $x\in \Omega$. Above, $\Phi$ denotes the cdf of the one-dimensional standard normal distribution $\normal(0,1)$, of density $\varphi$.
	
	Using our approach we can already obtain weighted Poincar\'e inequalities for $\mu$ based on the normal distribution $\mu_\Sigma$. For instance, as a combination of inequality \eqref{eq: weighted Poincare inequality normal gaussian} for $\mu_\Sigma$ and
	Proposition \ref{prop: copula Poincare inequality} we have
	\begin{equation}
		\label{eq: weighted Poincare inequality gaussian copula}
		\Var_\mu(f)\leq \int_{\Omega}\norm{D\,\nabla f }_\Sigma^2 \,d\mu.
	\end{equation}
	Here, $D$ is the diagonal matrix with entries $D_{i,i}=(\varphi\circ \Phi^{-1}\circ \FX{i})/\rho_i$. This inequality will be applied in the next section in the context of global sensitivity analysis.
	\\
	Since the weight $\Sigma$ is constant, we also obtain the following classical Poincar\'e inequality:
	\begin{equation}
		\label{eq: classical Poincare inequality constant gaussian copula Lipschitz constant}
		\Var_\mu(f) 
		\leq \lambda_{\max}(\Sigma) \int_{\Omega} \sum_{i=1}^n \left(\frac{\varphi\circ \Phi^{-1}\circ \FX{i}}{\rho_i}\,\frac{\partial f}{\partial x_i}\right)^2 \,d\mu
		\leq\lambda_{\max}(\Sigma)\, \max_{i}K_i^2\int_\Omega \norm{\nabla f}^2\,d\mu,
	\end{equation}
	provided the constants
	\[K_i=\sup_{x_i} \frac{\varphi\circ \Phi^{-1}\circ \FX{i}(x_i)}{\rho_i(x_i)}=\sup_{\tildx_i} \frac{\varphi(\tildx_i)}{\rho_i\circ \Phi \circ F_i^{-1}(\tildx_i)},\]
	are finite. 
	According to \eqref{eq: diagonal term D}, each $K_i$ is the Lipschitz constant $\norm{T_i}_{\Lip
	}$ of the optimal transport map from the standard normal distribution to $\mu_i$, given by $T_i=F^{-1}_i\circ \Phi$. As such, although
	$K_i$ does not admit a closed-form expression in general, useful bounds are available under log-concavity assumptions on the marginal $\mu_i$. A fundamental result is Caffarelli's contraction theorem (see \cite{caffarelli}). It states that if $\nu\in \probClass{\Omega}$, with $\Omega$ convex, is a probability measure with density $\rho_\nu=e^{-V_\nu}$ such that $\Hess(V_\nu)\geq \kappa\,\Id$ for some $\kappa>0$, then the optimal transport mapping from the standard multivariate normal distribution to $\nu$ has Lipschitz constant bounded from above by $1/\sqrt{\kappa}$. 
	
	As a direct consequence of Caffarelli's contraction theorem combined with \eqref{eq: classical Poincare inequality constant gaussian copula Lipschitz constant} for each marginal distribution $\mu_i$, one obtains the following result for probability measures with Gaussian copulas and whose marginals are uniformly log-concave.
	\begin{prop}
		Let $\mu\in \probClass{\Omega}$ be a probability distribution with Gaussian copula $\cop_\Sigma$, with $\Sigma\in \matClassR^n$. Assume that each marginal distribution $\mu_i$ is uniformly log concave, that is, with density $\rho_i=e^{-V_{i}}$ satisfying $V_{i}''\geq \kappa_i>0$. Then $\mu$ satisfies a classical Poincar\'e inequality with optimal constant such that
		\begin{equation*}
			C_P(\mu)\leq \lambda_{\max}(\Sigma) \max_{i} \frac{1}{\kappa_i}.
		\end{equation*}
	\end{prop}
	
	We have two remarks regarding this result. First, the bound essentially requires that each mapping $T_i$ is Lipschitz, and the uniform log-concavity of the marginals is a sufficient condition ensuring this property. This is equivalent to requiring that the Nataf transformation in \eqref{eq: Nataf transformation} is Lipschitz, for which $\norm{T}_{\Lip}=\max_i \norm{T_i}_{\Lip}$. Thus under this weaker condition, inequality \eqref{eq: classical Poincare inequality constant gaussian copula Lipschitz constant}, as well as similar inequalities beyond the Gaussian copula setting, follow directly from transport arguments. Second, if $\mu$ was uniformly log-concave, then a classical Poincar\'e inequality could be obtained directly using Caffarelli's contraction theorem applied to $\mu$, or Bakry-Emery criterion \eqref{eq: Bakry-Emery}. However, recall that our assumption of uniform log-concave marginals alone does not imply the log-concavity of the joint measure $\mu$. This raises the natural question of whether the additional assumption of having a Gaussian copula is sufficient to ensure such a property.
	This is not the case, as we show in the next example.
	
	Consider $\mu\in \R^2$ a probability measure with Gaussian copula $\cop_\Sigma$, where
	\[\Sigma=\begin{pmatrix} 1 & \ell \\ \ell & 1\end{pmatrix},\qquad \ell\in (-1,1).\]
	Given $\veps>0$, denote the function $h_\veps(x_i)=\sinh(x_i)+\veps\,x_i$ and assume that both marginals $\mu_i$ are given by
	\[\scaleobj{.92}{F_i(x_i)=\Phi\circ h_\veps(x_i),\quad \rho_i(x_i)=h_\veps'(x_i)\,\varphi\circ h_\veps(x_i)=\frac{1}{\sqrt{2\pi}}\left(\cosh(x_i)+\veps\right)\exp\left(-\frac{1}{2} \left(\sinh(x_i)+\veps\,x_i\right)^2 \right),}\]
	for all $x_i\in \R$. With this choice, each marginal $\mu_i$ is uniformly log-concave since its potential $V_i=-\log(\rho_i)$ satisfies
	\[V_i''(x_i) \geq 2\veps + \veps^2 > 0,\quad x_i \in \R.\]\\
	However, the joint measure $\mu$ is not log-concave for $\veps$ small enough. Indeed, using \eqref{eq: gaussian copula density} its density is given by
	\[
	\rho(x)
	= \frac{1}{2\pi\,\sqrt{1-\ell^2}}\, h_\veps'(x_1)h_\veps'(x_2)
	\exp\left(-\frac{1}{2(1-\ell^2)}(h_\veps(x_1)^2 + h_\veps(x_2)^2 - 2\,\ell\, h_\veps(x_1)h_\veps(x_2))\right),
	\]
	for all $x\in \R^2$, and a direct computation shows that the potential $V=-\log(\rho)$ is such that
	\[
	\det(\Hess (V)(0,0))
	= \left(\left(\frac{1}{1-\ell^2}-\frac{1}{(1+\veps)^3}\right)^2
	- \left(\frac{\ell}{1-\ell^2}\right)^2 \right)(1+\veps)^4.
	\]
	But this quantity is negative as soon as $(1+\veps)^3 < 1 + \abs{\ell}$, meaning that therefore
	the measure $\mu$ is not log-concave.
	
	\section{Application to Global Sensitivity Analysis}
	\label{section: Application to Global Sensitivity Analysis}
	\subsection{Sobol indices, DGSM, and their link via weighted Poincar\'e inequalities}
	In this section we apply weighted Poincar\'e inequalities to Global Sensitivity Analysis (GSA). To contextualize this connection, recall that the aim of GSA is to quantify the influence of input random variables $X=(X_1,\dots,X_d)$
	on the output of a function $f\colon \R^d \rightarrow \R$, which represents a 
	black-box model. The role of Poincar\'e inequalities in this applied domain is to provide a link between two commonly used sensitivity indices to quantify uncertainty: Sobol indices
	and Derivative-based Global Sensitivity Measures (DGSM).
	
	This connection has been extensively studied
	in the setting where the input variables are
	mutually independent, allowing the use of one-dimensional Poincar\'e inequalities. See for instance \cite{SobolKucherenko2009,lamboni,poincareintervals} for applications involving classical (unweighted) inequalities and \cite{Song,HerediaWeightPoincare} for more recent works using weighted Poincar\'e inequalities. Beyond the one-dimensional framework, the application of classical multidimensional Poincar\'e inequalities in GSA was introduced in the PhD thesis \cite{steiner}. This extension allows one to relax the assumption of independence between individual variables by considering instead that the inputs are given by independent random vectors. This setting, which can also be seen as having block-wise independent variables, is the one considered in the present work.

	To formalize this idea, consider $\parti =\{I_1,\dots,I_m\}$ a partition of the set $\{1,\dots,d\}$, where each $I_k$ is referred to as a block. For each block $I_k$, let $X_{I_k}=(X_i)_{i\in I_k}$ denote corresponding  sub-vector of inputs, with distributions $\mu_{I_k}\in \probClass{\Omega_{I_k}}$ defined on an open connected set $\Omega_{I_k}$ of class $\C^1$. Assume that these sub-vectors $X_{I_1},\dots,X_{I_m}$ are independent. For any collection of blocks $\I \subset \parti$, we denote $X_{\I}$ the concatenation of the vectors $(X_I)_{I \in \I}$. In particular the complete input vector is given by $X = X_{\parti}$.
	
	Let us introduce the devoted sensitivity indices for each $X_{I_k}$. First, consider total Sobol indices. They are defined through the Hoeffding-Sobol decomposition, which in the present block-wise independence setting takes the form
	\begin{equation*}
		f(X) = \sum_{\I \subset \mathcal{S}} f_{\I}(X_{\I})
	\end{equation*}
	(see \eg \cite{group_independent_GSA}).
	Each term $f_{\I}(X_{\I})$ is uniquely characterized by the property
	\begin{equation*}
		\mathbb{E}[f_{\I}(X_{\I}) \mid X_{\mathcal{J}}] = 0, \quad \text{for all } \mathcal{J} \subsetneq \I,
	\end{equation*}
	where by convention, $\mathbb{E}[\, \cdot \mid X_{\mathcal{J}}] = \mathbb{E}[\, \cdot \,]$ when $\mathcal{J} = \emptyset$. This property 
	leads to the following decomposition of the output variance:
	\begin{equation*}
		\Var(f(X)) = \sum_{\I \subset \parti} \Var(f_{\I}(X_{\I})).
	\end{equation*}
	Then the total Sobol index of each vector $X_{I_k}$ is obtained as the proportion of variance given by the elements of the decomposition where it appears, namely
	\begin{equation*}
		\SKtot = \frac{1}{\Var(f(X))} \sum_{\I \subset \parti, I_k \in \I} \Var(f_{\I}(X_{\I}))\in [0,1].
	\end{equation*}
	These indices can be seen as percentages measuring the degree of influence of each input vector. Despite their strong interpretability, total Sobol indices are numerically expensive to estimate. 
	
	When the derivatives of the model $f$ are available we can also use DGSM indices. We consider multivariate weighted versions of them. Namely, the weighted DGSM of $X_{I_k}$ with weight $W_{I_k}\in \weightClass{\Omega_{I_k}}$ is given by
	\[\DGSMk = \Esp\left[\norm{\nabla_{I_k}f(X)}_{W_{I_k}(X_{I_K})}^2\right],\quad \mbox{where}\quad \nabla_{I_k}f=\left(\frac{\partial f}{\partial x_i}\right)_{i\in I_k}.
	\]
	This definition generalizes both the classical DGSM, recovered when $W_{I_k}=\Id$, and the weighted DGSM for individual input variables considered in \cite{Song,HerediaWeightPoincare}, which is restricted to one-dimensional weights.
	
	DGSM indices are rarely used on their own to quantify the influence of input variables. Instead, they are typically used for screening purposes, in order to detect non-influential inputs. This is possible because they provide cost-effective upper bounds on total Sobol indices.
	These bounds, obtained as a direct consequence of weighted Poincar\'e inequalities, are stated formally in the following proposition. Its proof is a straightforward adaptation of the one in \cite{HerediaWeightPoincare}. 
	\begin{prop}
		\label{prop: upper bound Sobol DGSM}
		Under the same notation and assumptions above, assume that the distribution of the random vector $X_{I_k}$ satisfies a weighted Poincar\'e inequality with weight $\Weight_{K_I}\in \weightClass{\Omega_{I_k}}$.
		Then we have the inequality
		\begin{equation}
			\label{eq: prop upper boundd}
			\SKtot\leq C_P(\mu_{I_k},\Weight_{I_k}) \, \frac{\DGSMk}{\Var(f(X))}.
		\end{equation}
	\end{prop}
	This inequality provides a practical screening criterion: when estimations of the right-hand side in \eqref{eq: prop upper boundd} falls below a small threshold, $X_{I_k}$ can be identified as not influential without having to compute $\SKtot$. The variance of $f(X)$ and $\DGSMk$ can be estimated efficiently employing well-established techniques. For instance, Monte Carlo integration, using the available model and gradient evaluations,
	already provides reasonable approximations.
	\\
	Regarding the estimation of the optimal Poincar\'e constant $C_P(\mu_{I_k}\Weight_{I_k})$, 
	one possible approach is to adapt the finite element-based method proposed in \cite{steiner}, originally designed for the unweighted case $\Weight_{I_k}=\Id$, by incorporating weights. Nevertheless, this approach becomes computationally demanding as the dimension increases, and is specific to each probability measure as one has to deal with the geometry of the domain $\Omega_{I_k}$. In our numerical applications, however, such estimation is not required, as we use weighted Poincar\'e inequalities with explicit optimal constants.

	\subsection{Numerical application}
	We present numerical experiments illustrating the performance of the DGSM-based upper bounds on total Sobol indices in Proposition \ref{prop: upper bound Sobol DGSM}. To this end, we consider a dyke-flood toy model, a standard benchmark in the literature for testing GSA methodologies. It was also considered, for instance, in \cite{poincareintervals,steiner,HerediaWeightPoincare}. The outputs of interest in this model are the maximal annual overflow (measured in meters)
	\begin{equation*}
		S=Z_v-H_d-C_b+\left(\frac{Q}{B K_s}\sqrt{\frac{L}{Z_m-Z_v}}\right)^{\frac{3}{5}},
	\end{equation*}
	and the annual cost of maintenance of a dyke built next to it (measured in million of euros)
	\begin{equation*}
		C=\mathbbm{1}_{S>0}+\left(0.2+0.8\left(1-e^{-\frac{1000}{S^4}}\right)\right)\mathbbm{1}_{S\leq 0}+\frac{1}{20}\max\left\{ H_d,8\right \}.
	\end{equation*}
	The $d=8$ input variables appearing in both expressions are assumed to be block-wise independent. Both the block structure and the marginal distribution, as specified in Table \ref{table: flood model variables}, are consistent with the physical meanings of the variables. We preserve the same block structure as in \cite{steiner}, where the upper bounds \eqref{eq: prop upper boundd} were applied in the context of classical Poincar\'e inequalities, and under the assumption that the copulas of the paired  input variables are Gaussian.
	\begin{table}[ht]
		\centering
		\bgroup
		\def\arraystretch{1.3}
		\begin{tabular}{cllll}
			\hline
			Block & Input & Meaning & Unit & Probability measure \\ \hline 
			\multirow{2}{*}{$I_1=\{1,2\}$} & $X_1=Q$   & Max. flow rate & $m^3/s$ & Gumbel $\mathcal{G}(1013, 558)|_{[500, 3000]}$ \\  
			& $X_2=K_s$ & Strickler coeff. & --- & Normal $\normal(30,64)|_{[15,75]}$ \\ \hline
			\multirow{2}{*}{$I_2=\{3,4\}$} & $X_3=Z_v$ & Downstream level & $m$ & Triangular $\mathcal{T}(49, 50, 51)$\\ 
			& $X_4=Z_m$ & Upstream level & $m$ & Triangular $\mathcal{T}(54, 55, 56)$ \\ \hline
			$I_3=\{5\}$                  & $X_5=H_d$ & Dyke height & $m$ & Uniform $\mathcal{U}(7,9)$\\ \hline
			$I_4=\{6\}$                  & $X_6=C_b$ & Bank height & $m$ & Triangular $\mathcal{T}(55, 55.5, 56)$ \\ \hline
			\multirow{2}{*}{$I_5=\{7,8\}$} & $X_7=L$   & River length & $m$ &  Triangular $\mathcal{T}(4990, 5000, 5010)$ \\ 
			& $X_8=B$   & River width & $m$ & Triangular $\mathcal{T}(295, 300, 305)$ \\ \hline
		\end{tabular}
		\egroup
		\caption{Input variables of the flood model and their associated block structure. The notation $|_I$ means that the distribution is truncated on the set $I$.}
		\label{table: flood model variables}
	\end{table}\\
	The symbols $\mathcal{G}(\eta,\beta)$ ($\eta\in \R$, $\beta>0$) and $\mathcal{T}(a,c,b)$ ($a<c<b$) refer respectively to the Gumbel and triangular distribution. Their marginals are given by:
	\begin{equation*}
		\rho(x)=\frac{1}{\beta}\exp\left(-\frac{x-\eta}{\beta}-\exp\left(-\frac{x-\eta}{\beta}\right)\right),\qquad x\in \R,
	\end{equation*}
	and
	\begin{equation*}
		\rho(x)=\frac{2(x-a)}{(b-a)(c-a)}\mathbbm{1}_{[a,c]}(x)+\frac{2(b-x)}{(b-a)(b-c)}\mathbbm{1}_{(c,b]}(x),\qquad x\in \R.
	\end{equation*}
	Regarding the structure of dependence of the coupled variables, we consider two scenarios in which all vectors are modeled using copulas of the same type, selected among those for which weighted Poincar\'e inequalities are available:
	\begin{itemize}
		\item We consider a setting where dependence is fully modeled by Gaussian copulas $\cop_\Sigma$. Since each coupled vector is two-dimensional, the associated matrices $\Sigma$ take the form 
		$$\Sigma = \begin{pmatrix} 1 & \ell \\ \ell & 1 \end{pmatrix},\qquad \ell\in (-1,1).$$ 
		The parameters $\ell$, specified in Table \ref{tab:gaussian_copula}, are the same as those used in \cite{steiner}. The table also includes the associated Kendall rank correlation coefficient $\tau$ (simply called Kendall’s Tau), which, in the case of elliptical copulas, depends only on $\ell$ and is given by $\tau(\ell) = \frac{2}{\pi} \arcsin(\ell)$ (see \cite{tau_elliptical_copulas}). Recall that in general $\tau$ measures the dependence between marginals in terms of concordance: values close to $1$ (resp. to $-1$) indicate strong concordance (resp. discordance). In this case, the input vectors in the model present a moderate level of concordance.
		\begin{table}[ht]
			\centering
			\begin{tabular}{ccc}
				\hline
				Couple & Parameter $\ell$ &  $\ell$ \\
				\hline
				$X_{I_1}=(Q,K_s)$ & 0.5 & $\approx 0.33$ \\
				$X_{I_2}=(Z_v,Z_m)$ & 0.3 & $\approx 0.19$ \\
				$X_{I_6}=(L,B)$ & 0.3 & $\approx 0.19$ \\
				\hline
			\end{tabular}
			\caption{Gaussian copula parameters and corresponding Kendall's Tau.}
			\label{tab:gaussian_copula}
		\end{table}
		
		In this setting, for the upper bounds on total Sobol indices in \eqref{eq: prop upper boundd} we use the weighted Poincar\'e inequalities given in \eqref{eq: weighted Poincare inequality gaussian copula}. Recall that these remain valid for variables that are not grouped with other ones, which in this case are $X_3=H_d$ and $X_6=C_b$ (see \eqref{eq: copula weighted Poincare inequality 1D}). We refer to the corresponding upper bounds as \textit{weighted upper bounds}.\\
		For comparison purposes, we also replicate the computations of the \textit{Classical upper bounds} presented in \cite{steiner}, \ie, those using classical Poincar\'e inequalities.
		\item We consider a second setting where the dependence structure is encoded by Clayton copulas $\cop_\theta$. The parameters $\theta>0$ are specified in Table \ref{tab:clayton_copula}, together with their corresponding values of Kendall’s Tau, which is given in this case by 
		$\tau(\theta)=\theta/(\theta+2)$ (see \cite{copulas_harry}). The parameters are chosen so that these values are comparable to those of the Gaussian copula setting.
		\begin{table}[ht]
			\centering
			\begin{tabular}{ccc}
				\hline
				Couple & Parameter $\theta$ & $\tau(\theta)$ \\
				\hline
				$X_{I_1}=(Q,K_s)$ & 1.0 & $0.33$ \\
				$X_{I_2}=(Z_v,Z_m)$ & 0.5 & $0.20$ \\
				$X_{I_6}=(L,B)$ & 0.5 & $0.20$ \\
				\hline
			\end{tabular}
			\caption{Clayton copula parameters.}
			\label{tab:clayton_copula}
		\end{table}
		
		In this case we only compute \textit{weighted upper bounds}, given by the weighted Poincar\'e inequality in \eqref{eq: weighted Poincare inequality clayton copula}. Recall that the latter is based on the inequality  \eqref{eq: Poincare inequality inverted Dirichlet} for the inverted Dirichlet distribution. As previously mentioned, one could also rely on the alternative inequality for this measure given in \eqref{eq: Poincare inequality inverted Dirichlet not used}. However, since its Poincar\'e constant is not optimal, the accuracy of the numerical results would not be guaranteed.
	\end{itemize}
	
	Before presenting our results, we briefly discuss the numerical details. The implementation was carried out using the R statistical software \cite{R}. \medskip 
	
	\noindent \textit{Model scaling}
	\hfill\\
	To ensure numerical stability, all the input variables were standardized according to their parameters; for instance, the truncated normal variable $K_s$ was scaled by its mean and variance. The outputs $S$ and $C$ have been scaled accordingly. Since this scaling involves only monotonic transformations of the marginals, the copulas of the paired variables do not change (see \cite{Nelsen_copulas}). The total Sobol indices and their upper bounds also remain unchanged.
	\medskip 
	
	\noindent \textit{Numerical aspects of the Poincar\'e inequalities}
	\hfill\\
	As mentioned, the weighted upper bounds that we implement are based on the Poincar\'e inequalities in \eqref{eq: weighted Poincare inequality gaussian copula} and \eqref{eq: weighted Poincare inequality clayton copula}. As such, their computation only requires the densities $\rho_i$, the cdf $F_i$, and the quantile function $\Phi^{-1}$ of the standard normal distribution, appearing in the inequalities. All these functions are available in R. \\
	For the classical upper bounds in the Gaussian copula setting, we require numerical estimations of the optimal constants $C_P(\mu_{I_k})$ of the classical Poincar\'e inequalities. These are computed using a finite elements discretization. See \cite{poincareintervals} for the method in the one-dimensional case and \cite{steiner} for the multidimensional setting, including a special treatment for distributions with Gaussian copulas. 
	
	\medskip 
	\noindent \textit{Monte Carlo estimation and sampling}
	\hfill\\
	The corresponding weighted DGSM and the variance term in the upper bound \eqref{eq: prop upper boundd} were estimated via Monte Carlo integration. We employed a sample of size $10*d=80$ of the input variables, together with the corresponding model and gradient evaluations. These gradient evaluations are in fact approximations using finite differences.
	\\
	The chosen sampling method of the input vectors depends on the copula. For distributions with a Gaussian copula $\cop_\Sigma$, the samples were obtained by generating normal random vectors with covariance matrix $\Sigma$ and applying the Nataf transformation \eqref{eq: Nataf transformation}. For the Clayton copula setting, samples were generated using the Marshall-Olkin algorithm for Archimedian copulas (see \eg \cite{copula_simulation}).
	\\
	We perform $100$ replicates for each estimation of the upper bounds, using newly generated samples at each run. They are displayed with boxplots to represent confidence intervals. In addition, to evaluate the accuracy of the upper bounds, the values of the total Sobol indices are required. Since analytical expressions are not available, we estimate them using the function {\fontfamily{qcr}\selectfont soboljansen} from the package {\fontfamily{qcr}\selectfont sensitivity} \cite{R_sensitivity}, through a sample of size 40.000. These estimations are taken as the ``true'' values.
	
	\medskip
	We first present the results for the Gaussian copula setting. They are shown in Figure \ref{fig: results Gaussian copula setting}, for the two outputs $S$ and $C$ simultaneously. 
	We can observe that for each block of variables and each model output, the weighted upper bounds improve the classical ones. Moreover, excepting for the individual variable $H_d$ in the cost output $C$, these bounds are very accurate.\\ A possible explanation for this outcome could be given in terms of the model's main effects, which are functions representing the model's behavior with respect to each input vector. As discussed in \cite{HerediaWeightPoincare} in the context of one-dimensional Poincar\'e inequalities, the accuracy of the upper bounds is related to the proximity of the main effects to the functions that attain equality in the associated Poincar\'e inequalities. In the one-dimensional setting, these extremal functions are relatively easy to identify, since they are characterized as the only strictly monotone eigenfunctions of the operators associated with the inequalities. However, we do not explore this direction in the present paper, since no analogous characterisation is known in the multidimensional setting.
	\begin{figure}[h!t]
		\centering
		\begin{subfigure}{.48\textwidth}
			\centering
			\includegraphics[width=1\linewidth]{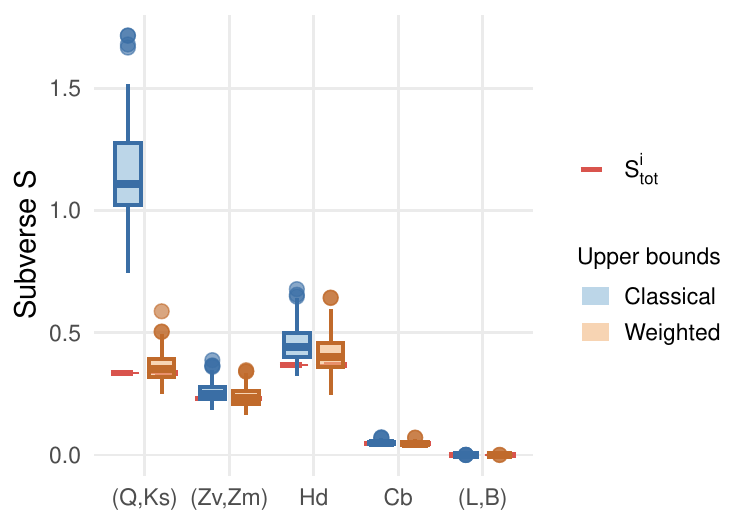}
		\end{subfigure}
		\begin{subfigure}{.48\textwidth}
			\centering
			\includegraphics[width=1\linewidth]{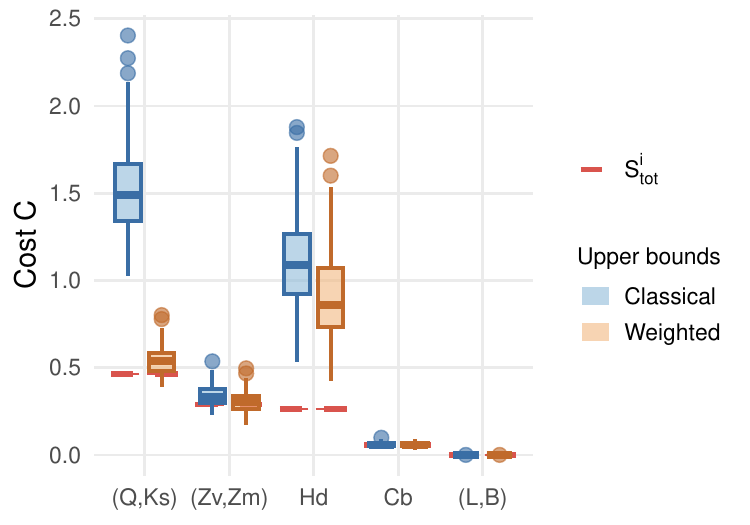}
		\end{subfigure}
		\caption{Upper bounds on the total Sobol indices for the flood model in the Gaussian copula setting. Horizontal bars indicate the true values.}
		\label{fig: results Gaussian copula setting}
	\end{figure}
	
	Next, Figure \ref{fig: results Clayton copula setting} presents the results for the Clayton copula setting. Most of the upper bounds are accurate, with the exception of those for the variables $(Q,K_s)$ in both outputs and, as in the previous case, for $H_d$ in the cost output $C$. The presence of outliers in the boxplots is due to the terms $T_i^{-1}(x_i)=F_i(x_i)^{-\theta}-1$ in the weighted Poincar\'e inequality \eqref{eq: weighted Poincare inequality clayton copula}, which become very large when $F_i(x_i)$ takes values close to zero.
	
	Better estimations of total Sobol indices can be obtained using more sophisticated techniques, still relying on information on the gradient. Staying within the scope of Poincar\'e inequalities, a natural next step for this work would be to consider Poincar\'e chaos expansions. Roughly speaking, the current upper bounds only use a small portion of the so-called Poincar\'e basis associated with the inequalities, whereas chaos expansions involve the full basis when it exists. For a detailed discussion on this approach in the context of mutually independent input variables, see \cite{gradient_enhanced,PoincareChaos}.
	
	\begin{figure}[h!t]
		\centering
		\begin{subfigure}{.48\textwidth}
			\centering
			\includegraphics[width=1\linewidth]{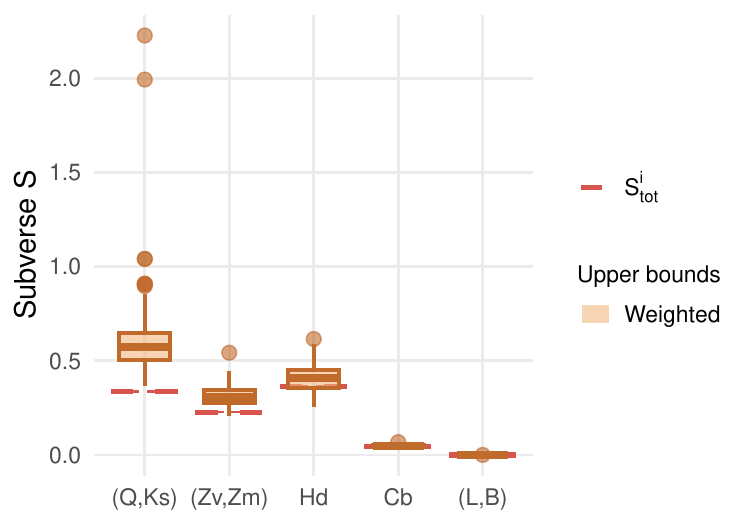}
		\end{subfigure}
		\begin{subfigure}{.48\textwidth}
			\centering
			\includegraphics[width=1\linewidth]{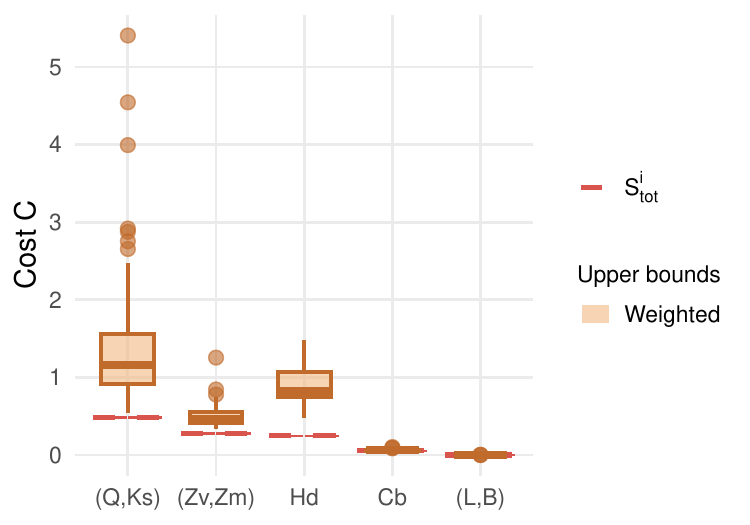}
		\end{subfigure}
		\caption{Upper bounds on the total Sobol indices for the flood model in the Clayton copula setting. Horizontal bars indicate the true values.}
		\label{fig: results Clayton copula setting}
	\end{figure}
	
	\appendix
	\section{Weighted Poincar\'e inequalities for one-dimensional probability measures}
	
	In this appendix we estimate the optimal Poincar\'e constant for two one-dimensional probability measures, namely the Beta distribution and a logistic-type distribution. We establish upper bounds on the Poincar\'e constant using the intertwining approach introduced in \cite{bj} and lower bounds with the Rayleigh quotient characterization of the optimal constant. 
	Let us recall this strategy briefly. Additional examples of application can be found in \cite{bj,bjm,HerediaWeightPoincare}.
	
	Let $\mu\in \probClass{I}$ be a one-dimensional probability measure defined on an open interval $I$ and $w\in \weightClass{I}$ be a weight. Consider the  (self-adjoint extension of the) diffusion operator defined for smooth functions $f$ as 
	\[L f := \frac{1}{\rho}(w\,f'\,\rho)'.\]
	There are several characterizations of the optimal Poincar\'e constant $C_P(\mu,w)$ and ways to bound it in terms of this operator. For instance recall the intertwining result in \cite{bj}. Given a function $g$ such that $\abs{g'}>0$ on $I$, define
	\begin{equation}
		\label{eq: intertwining function M}
		M_g=\frac{(-Lg)'}{g'}.
	\end{equation}
	If $g$ is such that $M_g$ is bounded from below by a positive constant on $I$, then $\mu$ satisfies the following weighted Poincar\'e inequalities:
	\[\Var_\mu(f)\leq \int_{I}\frac{w}{M_g} (f')^2 \,d\mu\leq \frac{1}{\inf_I \,M_g} \int_{I}w\, (f')^2 \,d\mu.\]
	In particular, the constant $C_P(\mu,w)$ is bounded from above by
	\begin{equation}
		\label{eq: intertwining inequality}
		C_P(\mu,w)\leq \frac{1}{\inf_I M_g}.
	\end{equation}
	Hoping to obtain a sharp bound, in practice one typically selects a parameterized family of functions for $g$ and then optimizes the right-hand side or \eqref{eq: intertwining inequality} with respect to the corresponding family of parameters.
	
	Regarding lower bounds, we use the following characterization of $C_P(\mu,w)$, which follows directly from definition of the optimal Poincar\'e constant:
	\begin{equation}
		\label{eq: Rayleigh quotient Poincare}
		C_P(\mu,w)=\sup\left\{ \frac{\Var_\mu(f)}{\int_I w\,(f')^2\, d\mu} \, \big| \, f\in L^2(\mu),\; 0<\int_I w\,(f')^2\,d\mu<+\infty \right\}.
	\end{equation}
	\mathTitle{Beta distribution}
	Consider the beta distribution $\Pi_\beta$ with vector parameter $\beta=(\beta_1,\beta_2)$, where $\beta_1,\beta_2>0$. Its density is defined for all $t\in (0,1)$ as $\rho(t)=Z_{\beta_1,\beta_2}^{-1}\, t^{\beta_1-1}(1-t)^{\beta_2-1}$, where  $Z_{\beta_1,\beta_2}=\frac{\Gamma(\beta_1)\,\Gamma(\beta_2)}{\Gamma(\beta_1+\beta_2)}$. 
	Considering the weight $w_\tau(t)=t^2(1-t)^2$,
	here we show that $C_P(\Pi_\beta,w_\tau)$ is bounded from above by
	\begin{equation*}
		\Phi(\beta_1, \beta_2) = \begin{cases}
			\dfrac{4}{\min\left(\beta_1^2,\beta_2^2\right)}, & \text{if } \min(\beta_1, \beta_2) \leq 1, \\[12pt]
			\dfrac{4}{\min(2\beta_1 - 1,2\beta_2 - 1)}, & \text{if } \min(\beta_1, \beta_2) \geq 1,
		\end{cases}
	\end{equation*}
	and bounded from below by $4/\min(\beta_1^2,\beta_2^2)$. Note then that the bounds are sharp when $\min(\beta_1,\beta_2)\leq 1$, meaning that $C_P(\Pi_\beta,w_\tau)=4/\min(\beta_1^2,\beta_2^2)$ in this regime.
	
	First we establish the upper bound.
	Consider a function $g$ such that $ g'(t)=t^{\veps_1}(1-t)^{\veps_2}$, $t\in (0,1)$, with $\veps_1,\veps_2\in \R$. 
	After some computations we obtain the associated function $M_g$ in \eqref{eq: intertwining function M}, here given for all $t\in (0,1)$ by
	\begin{multline*}
		M_{g}(t)=-(\veps_1+1)(\beta_1+\veps_1+1)(1-t)^2\\+\left[(\beta_1+\veps_1+1)(\veps_2+2)+(\beta_2+\veps_2+1)(\veps_1+2)\right]t(1-t)-(\veps_2+1)(\beta_2+\veps_2+1)t^2.
	\end{multline*}
	Since $M_g$ is a degree two polynomial, 
	its minimum depends on the sign of its leading coefficient. After expanding into the standard form $M(t)=A\,t^2+B\,t+C$, this coefficient is found to be
	\begin{align*}
		A			&=-(\veps_1+\veps_2+3)(\beta_1+\beta_2+\veps_1+\veps_2+2). 
	\end{align*}
	As soon as $A\leq 0$, the minimum of $M_g$ is reached at $\{0,1\}$ and it is thus given by
	\begin{equation}
		\label{eq: intertwining minimum beta function M}
		\min_{[0,1]} M_g = \min(M_g(0),M_g(1))=\min(-(\veps_1+1)(\beta_1+\veps_1+1),-(\veps_2+1)(\beta_2+\veps_2+1)).
	\end{equation}
	Seeking this quantity to be positive,
	we impose the additional assumptions
	\[(\veps_1+1)(\beta_1+\veps_1+1)<0,\quad (\veps_2+1)(\beta_2+\veps_2+1)<0.\]
	Let us explore scenarios with specific choices of $\veps_1$ and $\veps_2$.
	\begin{enumerate}
		\item First, if we take $\veps_1=\veps_2=-\frac{3}{2}$, then $A=0$ and from \eqref{eq: intertwining minimum beta function M} and \eqref{eq: intertwining inequality} it follows that
		\[C_P(\Pi_\beta,w_\tau)\leq 
		\frac{4}{\min\left(2\beta_1-1,2\beta_2-1\right)},\]
		whenever $\beta_1,\beta_2>\frac{1}{2}$.
		\item Then we consider $\veps_1=-\beta_1/2-1$ and $\veps_2=-\beta_2/2-1$. Thus the condition $A\leq 0$ is equivalent to $\beta_1+\beta_2\leq 2$. Under this constraint we obtain
		\[C_P(\Pi_\beta,w_\tau)\leq  \frac{4}{\min(\beta_1^2,\beta_2^2)}.\]
		\item Selecting the parameters $\varepsilon_1=-\beta_1/2-1$ and $\varepsilon_2=-\frac{3}{2}$, we have $A=\frac{1}{4}(\beta_1-1)\left(\beta_1+2\beta_2-1\right)$ and 
		\[C_P(\Pi_\beta,w_\tau)\leq \frac{4}{\min(\beta_1^2,2\beta_2-1)},\]
		under the conditions $\beta_2>1/2$ and $A\leq 0$, where the latter simplifies to $\beta_1\leq 1$.
		
		Exchanging the roles of $\veps_1$ and $\veps_2$, we also deduce that when $\beta_2\leq 1$,
		\[C_P(\Pi_\beta,w_\tau)\leq \frac{4}{\min(2\beta_1-1,\beta_2^2)}.\]
	\end{enumerate}
	Note that 
	if $\beta_1$ and $\beta_2$ are such that $\beta_1,\beta_2>1/2$ and $\beta_1+\beta_2\leq 2$, we have obtained different upper bounds (given in the first and the second cases above). We preserve the smallest one, being $4/\min(\beta_1^2,\beta_2^2)$. Additionally, note that if $\beta_1\leq 1<\beta_2$ (or if $\beta_2\leq 1<\beta_1$), we fall in the third case and the corresponding bound can also be written as
	\[\frac{4}{\min(\beta_1^2,2\beta_2-1)}=\frac{4}{\beta_1^2}=\frac{4}{\min(\beta_1^2,\beta_2^2)}.\]
	In summary, we have proven $C_P(\Pi_\beta,w_\tau)\leq \Phi(\beta_1,\beta_2)$.
	
	We have not explored all possible combinations of $\veps_1$ and $\veps_2$ in order to determine the smallest upper bound, since the computations involved are not straightforward. Our particular choices of $\veps_1$ and $\veps_2$ are inspired by the analysis in \cite{HerediaWeightPoincare} that led to the identification of the optimal Poincar\'e constant $C_P(\Pi_\beta,w_\tau)$ in the symmetric case $\beta_1=\beta_2$. In our non-symmetric case, the bound obtained also coincides with the optimal constant in the regime $\min(\beta_1,\beta_2)\leq 1$, for which we then have $C_P(\Pi_\beta,w_\tau)=4/\min(\beta_1^2,\beta_2^2)$.	Indeed, consider the function given by $ h(t)=t^\eta$ for all $t\in (0,1)$, with $\eta>-\beta_1/2$, so that $h\in L^2(\Pi_\beta)$ and $\int_0^1 w_\tau \,(h')^2\,d\Pi_\beta<+\infty$. We compute the Rayleigh quotient:
	\begin{equation*}
		\frac{\Var_{\Pi_\beta}(h)}{
			\int_{0}^1 w_\tau\;\left(h'\right)^2\,d\Pi_\beta }=\frac{1}{\eta^2}
		\frac{ Z_{\beta_1+2\eta,\beta_2}-\frac{1}{ Z_{\beta_1,\beta_2}} Z_{\beta_1+\eta,\beta_2}^2 }{ Z_{\beta_1+2\eta,\beta_2+2}}.
	\end{equation*}
	Using properties of the Gamma function, we have
	\[Z_{\beta_1+2\eta,\beta_2+2}=R(\beta_1,\beta_2,\eta)\, Z_{\beta_1+2\eta,\beta_2},\quad \mbox{where}\quad R(\beta_1,\beta_2,\eta)=\frac{\beta_2}{\beta_1+2\eta+\beta_2}\,\times\,\frac{\beta_2+1}{\beta_1+2\eta+\beta_2+1}.\]
	Then
	\[\frac{\Var_{\Pi_\beta}(h)}{
		\int_{0}^1 w_\tau\left(h'\right)^2\,d\Pi_\beta }=\frac{1}{\eta^2\,R(\beta_1,\beta_2,\eta)}
	\,	\left(1-\frac{Z_{\beta_1+\eta,\beta_2}^2}{Z_{\beta_1+2\eta,\beta_2}\times Z_{\beta_1,\beta_2}}\right)\]
	The term $Z_{\beta_1+2\eta,\beta_2}$ in the denominator tends to infinity as $\eta\rightarrow -\beta_1/2$ and $R(\beta_1,\beta_2,\eta)$ tends to one. Hence, due to the characterization of the Poincar\'e constant in \eqref{eq: Rayleigh quotient Poincare}, after taking the limit $\eta\rightarrow -\beta_1/2$ we obtain $C_P(\Pi_\beta,w_\tau)\geq 4/\beta_1^2$. Analogous computations using the function $t\mapsto h_\eta(t)=(1-t)^\eta$ with $\eta\rightarrow -\beta_2/2$ lead to $C_P(\Pi_\beta,w_\tau)\geq 4/\beta_2^2$, therefore that $C_P(\Pi_\beta,w_\tau)\geq 4/\min(\beta_1^2,\beta_2^2)$.
	
	\mathTitle{A logistic-type distribution}
	This part is devoted to estimating the optimal Poincar\'e for the logistic-type distribution $\muR\in \probClass{\Rplus}$ having density function
	\[r\in\Rplus\mapsto \rho(r)=2\,c_{a}^{-1}\,r^{2a-1}e^{-r^2}(1+e^{-r^2})^{-2},\quad \mbox{where }c_a=\int_{\Rplus} r^{a-1}e^{-r}(1+e^{-r})^{-2}\,dr,\; a>0,\]
	with weight $\wR(r)=r^2$. To this end, we first consider the auxiliary measure $\mutild$ with density
	\[\rhotild(r)=c_{a}^{-1}\,r^{a-1}e^{-r}(1+e^{-r})^{-2},\quad r\in \Rplus,\]
	and provide estimations of $C_P(\mutild,\wR)$. Then bounds for $C_P(\muR,\wR)$ are recovered via a transport argument.
	
	Let us show that $C_P(\mutild,\wR)$ is bounded from above by
	\[\phi(a)=\begin{cases}
		\dfrac{4}{a^2},&\mbox{if }a\leq 2, \\[8pt]
		\dfrac{1}{a-1}, & \mbox{if }a\geq 2,
	\end{cases}\]
	and from below by $4/a^2$. In particular the bounds are tight for $a\leq 2$, in which case we have $C_P(\mutild,\wR)=4/a^2$. Indeed, consider a function $g$ such that $g'(r)=r^\veps$ for all $r\in \Rplus$, with $\veps\in \R$. One can then check that the associated function $M_g$ in \eqref{eq: intertwining function M} can be written as
	\[M_g(r)=-(\veps+a+1)(\veps+1)+(\veps+2)\,r\left(2\frac{e^{r}}{1+e^{r}}-1\right)+2\,r^2\frac{e^r}{1+e^r}\left(1-\frac{e^r}{1+e^r}\right),\quad r\in\Rplus.\]
	The terms inside the parentheses are non-negative. Therefore, as soon as $\veps\geq -2$, the infimum of $M_g$ is reached at zero, in which case we obtain
	\[\inf_{\Rplus} M_g=-(\veps+a+1)(\veps+1).\]
	Maximizing this expression with respect to $\veps$ yields the optimal parameter
	$\veps^*=-\frac{1}{2}a-1$, with associated value $\inf_{\Rplus} M_g=a^2/4$. However, this is true only under the condition $a\leq 2$ ensuring $\veps^*\geq -2$. Otherwise, if $a>2$, the best parameter under the constraint $\veps\geq -2$ is $\veps^*=-2$, leading to $\inf_{\Rplus} M_g=a-1$. Summarizing these computations along with \eqref{eq: intertwining inequality} we conclude that $C_P(\mutild,\wR)\leq \phi(a)$.
	
	To obtain the lower bound we consider the function $r \mapsto h(r)=r^\eta$, with $\eta>-a/2$ so that $f\in L^2(\mutild)$ and $\int_{\Rplus} \wR \, (h')^2 \, d\mutild<\infty$. We compute
	\begin{equation*}
		\frac{\Var_{\mutild}(h)}{
			\int_{\Rplus} \wR\left(h'\right)^2\,d\mutild }=\frac{1}{\eta^2}
		\frac{ c_{2\eta+a}-\frac{1}{c_a} c_{\eta+a}^2 }{ c_{2\eta+a} }=\frac{1}{\eta^2}\left(1-\dfrac{c_{a+\eta}^2}{
			c_{2\eta+a}\times c_{a}}\right).
	\end{equation*}
	The term $c_{2\eta+a}$ diverges as $\eta\rightarrow -a/2$. Therefore due to \eqref{eq: Rayleigh quotient Poincare} taking the limit entails $C_P(\mutild,\wR)\geq 4/a^2$.
	
	Finally, we return to the measure $\muR$, which is the image of $\mutild$ under the mapping $r\mapsto T(r)=r^{\frac{1}{2}}$. Indeed, for every measurable, non-negative or bounded function $f\colon \Rplus\rightarrow \R$ we have
	\begin{multline*}
		\int_{\Rplus} f\circ T \,d\mutild=c_a^{-1}\int_{\Rplus} f(r^{\frac{1}{2}})\,r^{a-1} \frac{e^{-r}}{(1+e^{-r})^2}\,dr\\=2\,c_a^{-1}\int_{\Rplus} f(r)\,r^{2a-1} \frac{e^{-r^2}}{(1+e^{-r^2})^2}\,dr=\int_{\Rplus }f\,d\muR.
	\end{multline*}
	By the transport argument in \eqref{eq: weighted Poincare inequality by transport}, it follows that $\muR$ satisfies a weighted Poincar\'e inequality with weight $r\mapsto (\wR \times (T')^2)\circ T^{-1}(r)=r^2/4=\wR(r)/4$, preserving the optimal constant $C_P(\mutild,\wR)$. In other words, it satisfies
	\[\frac{1}{a^2}\leq C_P(\muR,\wR) \leq \frac{\phi(a)}{4}.\]
	
	\section*{Acknowledgement}
	I would like to express my sincere gratitude to my PhD advisor, Ald\'eric Joulin, for his invaluable guidance and suggestions, which greatly improved the presentation of this manuscript. I also thank Cl\'ement Steiner for the code that served as a starting point for the numerical applications.
	This work has been (partially) supported by the Project CONVIVIALITY ANR-23-CE40-0003 and DySLos ANR-25-CE40-6875-01 of the French National Research Agency.

	\bibliographystyle{plain}
	\bibliography{bibliography}

\begin{thebibliography}{10}

\bibitem{nataf_optimal}
A.~Alfonsi and B.~Jourdain.
\newblock A remark on the optimal transport between two probability measures
  sharing the same copula.
\newblock {\em Stat. Probab. Lett.}, 84:131--134, 2014.

\bibitem{ABJ_intertwining}
M.~Arnaudon, M.~Bonnefont, and A.~Joulin.
\newblock Intertwinings and generalized {Brascamp}-{Lieb} inequalities.
\newblock {\em Rev. Mat. Iberoam.}, 34(3):1021--1054, 2018.

\bibitem{BGL}
D.~Bakry, I.~Gentil, and M.~Ledoux.
\newblock {\em Analysis and Geometry of Markov Diffusion operators}, volume 348
  of {\em Grundlehren der mathematischen Wissenschaften}.
\newblock Springer, Heidelberg, 2013.

\bibitem{bobkov_spherically_symmetric}
S.~Bobkov.
\newblock {\em Spectral Gap and Concentration for Some Spherically Symmetric
  Probability Measures}, pages 37--43.
\newblock Springer Berlin Heidelberg, Berlin, Heidelberg, 2003.

\bibitem{bj}
M.~Bonnefont and A.~Joulin.
\newblock Intertwining relations for one-dimensional diffusions and application
  to functional inequalities.
\newblock {\em Potential Anal.}, 41(4):1005–1031, 2014.

\bibitem{bjm}
M.~Bonnefont, A.~Joulin, and Y.~Ma.
\newblock A note on spectral gap and weighted {P}oincar\'e inequalities for
  some one-dimensional diffusions.
\newblock {\em ESAIM Probab. Stat.}, 20:18--29, 2016.

\bibitem{joulin_bonnefont_spherically_symmetric}
M.~Bonnefont, A.~Joulin, and Y.~Ma.
\newblock Spectral gap for spherically symmetric log-concave probability
  measures, and beyond.
\newblock {\em J. Funct. Anal.}, 270(7):2456--2482, 2016.

\bibitem{brascamp_lieb}
H.~Brascamp and E.~Lieb.
\newblock On extensions of the brunn-minkowski and pr{\'e}kopa-leindler
  theorems, including inequalities for log concave functions, and with an
  application to the diffusion equation.
\newblock {\em J. Funct. Anal.}, 22(4):366--389, 1976.

\bibitem{group_independent_GSA}
B.~Broto, F.~Bachoc, M.~Depecker, and J.~Martinez.
\newblock Sensitivity indices for independent groups of variables.
\newblock {\em Math. Comput. Simulat.}, 163:19--31, 2019.

\bibitem{caffarelli}
L.~Caffarelli.
\newblock Monotonicity properties of optimal transportation and the {FKG} and
  related inequalities.
\newblock {\em Commun. Math. Phys.}, 214(3):547--563, 2000.

\bibitem{active_subspaces}
P.~Constantine, E.~Dow, and Q.~Wang.
\newblock Active subspace methods in theory and practice: applications to
  kriging surfaces.
\newblock {\em SIAM J. Sci. Comput.}, 36(4):a1500--a1524, 2014.

\bibitem{cui2024optimal}
T.~Cui, X.~Tong, and O.~Zahm.
\newblock Optimal {R}iemannian metric for {P}oincar{\'e} inequalities and how
  to ideally precondition {L}angevin dynamics.
\newblock {\em Preprint,
  \href{https://arxiv.org/abs/2404.02554}{arXiv:2404.02554}}, 2024.

\bibitem{meta_elliptical_distributions}
H.~Fang, K.~Fang, and S.~Kotz.
\newblock The meta-elliptical distributions with given marginals.
\newblock {\em J. Multivar. Anal.}, 94:1--16, 2002.

\bibitem{symmetric_multivariate_distributions}
K.~Fang, S.~Kotz, and K.~Ng.
\newblock {\em Symmetric multivariate and related distributions}, volume~36 of
  {\em Monogr. Stat. Appl. Probab.}
\newblock Chapman {and} Hall, 1990.

\bibitem{elliptical_distributions_in_statistics}
A.~Gupta, T.~Varga, and T.~Bodnar.
\newblock {\em Elliptically contoured models in statistics and portfolio
  theory}.
\newblock Springer, 2013.

\bibitem{multivariate_liouville_3}
R.~Gupta and D.~Richards.
\newblock Multivariate {Liouville} distributions, iii.
\newblock {\em J. Multivar. Anal.}, 43(1):29--57, 1992.

\bibitem{HerediaWeightPoincare}
D.~Heredia, A.~Joulin, and O.~Roustant.
\newblock On one dimensional weighted {P}oincar{\'e} inequalities for global
  sensitivity analysis.
\newblock {\em J. Math. Anal. Appl.}, 554(2), 2026.

\bibitem{huguet_cauchy}
B.~Huguet.
\newblock Poincar{\'e} inequalities and integrated curvature-dimension
  criterion for generalised {Cauchy} and convex measures.
\newblock {\em Bernoulli}, 30(3):2207--2227, 2024.

\bibitem{elliptical_distributions_dependence}
H.~Hult and F.~Lindskog.
\newblock Multivariate extremes, aggregation and dependence in elliptical
  distributions.
\newblock {\em Adv. Appl. Probab.}, 34(3):587--608, 2002.

\bibitem{R_sensitivity}
B.~Iooss, S.~{Da Veiga}, A.~Janon, and G.~Pujol.
\newblock sensitivity: Global sensitivity analysis of model outputs, 2023.
\newblock {R} package version 1.29.0.

\bibitem{copulas_harry}
H.~Joe.
\newblock {\em Dependence modeling with copulas}, volume 134 of {\em Monogr.
  Stat. Appl. Probab.}
\newblock CRC Press, 2014.

\bibitem{lamboni}
M.~Lamboni, B.~Iooss, A.~Popelin, and F.~Gamboa.
\newblock Derivative-based global sensitivity measures: general links with
  {Sobol} indices and numerical tests.
\newblock {\em Math. Comput. Simulat.}, 87:45--54, 2013.

\bibitem{stein_kernels_ley}
C.~Ley, G.~Reinert, and Y.~Swan.
\newblock Stein's method for comparison of univariate distributions.
\newblock {\em Probab. Surv.}, 14:1--52, 2017.

\bibitem{tau_elliptical_copulas}
F.~Lindskog, A.~McNeil, and U.~Schmock.
\newblock Kendall's tau for elliptical distributions.
\newblock In G.~Bol, G.~Nakhaeizadeh, S.~Rachev, T.~Ridder, and K.~Vollmer,
  editors, {\em Credit Risk}, pages 149--156, Heidelberg, 2003. Physica-Verlag
  HD.

\bibitem{robust_statistics}
R.~Maronna, R.~Martin, V.~Yohai, and M.~Salibi{\'a}n-Barrera.
\newblock {\em Robust statistics. {Theory} and methods (with {R})}.
\newblock John Wiley \& Sons, 2019.

\bibitem{inequalities_book}
A.~Marshall, I.~Olkin, and B.~Arnold.
\newblock {\em Inequalities: theory of majorization and its applications}.
\newblock Springer, 2011.

\bibitem{copula_simulation}
A.~McNeil.
\newblock Sampling nested {Archimedean} copulas.
\newblock {\em J. Stat. Comput. Simul.}, 78(5-6):567--581, 2008.

\bibitem{miclo_dirichlet}
L.~Miclo.
\newblock About projections of logarithmic {Sobolev} inequalities.
\newblock In {\em S\'eminaire de probabilit\'es XXXVI}, pages 201--221.
  Springer, 2003.

\bibitem{multivariate_pareto}
T.~Nayak.
\newblock Multivariate {Lomax} distribution: Properties and usefulness in
  reliability theory.
\newblock {\em J. Appl. Probab.}, 24:170--177, 1987.

\bibitem{Nelsen_copulas}
R.~Nelsen.
\newblock {\em An introduction to copulas.}
\newblock Springer, 2006.

\bibitem{dirichlet_book}
K.~Ng, G.~Tian, and M.~Tang.
\newblock {\em {Dirichlet} and related distributions: Theory, methods and
  applications}.
\newblock John Wiley \& Sons, 2011.

\bibitem{high-dimensional_matrix_estimation}
E.~Ollila, D.~Palomar, and F.~Pascal.
\newblock Shrinking the eigenvalues of {M}-estimators of covariance matrix.
\newblock {\em IEEE Trans. Signal Process.}, 69:256--269, 2021.

\bibitem{R}
{R Core Team}.
\newblock {R}: A language and environment for statistical computing, 2023.

\bibitem{poincareintervals}
O.~Roustant, F.~Barthe, and B.~Iooss.
\newblock Poincar\'e inequalities on intervals - application to sensitivity
  analysis.
\newblock {\em Electron. J. Stat.}, 11:3081--3119, 2017.

\bibitem{PoincareChaos}
O.~Roustant, F.~Gamboa, and B.~Iooss.
\newblock Parseval inequalities and lower bounds for variance-based sensitivity
  indices.
\newblock {\em Electron. J. Stat.}, 14:386--412, 2020.

\bibitem{gradient_enhanced}
O.~Roustant, N.~Lüthen, D.~Heredia, and B.~Sudret.
\newblock Gradient-enhanced global sensitivity analysis with {P}oincar{\'e}
  chaos expansions, 2025.

\bibitem{shimakura}
N.~Shimakura.
\newblock {\'E}quations diff{\'e}rentielles provenant de la g{\'e}n{\'e}tique
  des populations.
\newblock {\em Tohoku Math. J.}, 29(2):287 -- 318, 1977.

\bibitem{SobolKucherenko2009}
I.~Sobol and S.~Kucherenko.
\newblock Derivative-based global sensitivity measures and the link with global
  sensitivity indices.
\newblock {\em Math. Comput. Simulat.}, 79:3009--3017, 2009.

\bibitem{Song}
S.~Song, T.~Zhou, L.~Wang, S.~Kucherenko, and Z.~Lu.
\newblock Derivative-based new upper bound of {S}obol sensitivity measure.
\newblock {\em Reliab. Eng. Syst. Saf.}, 187:142--148, 2019.

\bibitem{steiner}
C.~Steiner.
\newblock {\em Sur l'utilisation des relations d'entrelacement dans l'{\'e}tude
  des g{\'e}n{\'e}rateurs de Markov auto-adjoints. : Application aux
  in{\'e}galit{\'e}s spectrales et fonctionnelles et {\`a} l'analyse de
  sensibilit{\'e}}.
\newblock Theses, {Universit{\'e} Paul Sabatier - Toulouse III}, 2022.

\bibitem{verdiere2025}
R.~Verdi{\`e}re, C.~Prieur, and O.~Zahm.
\newblock Diffeomorphism-based feature learning using {Poincar{\'e}}
  inequalities on augmented input space.
\newblock {\em J. Mach. Learn. Res.}, 26:31, 2025.

\bibitem{veysseire}
L.~Veysseire.
\newblock A harmonic mean bound for the spectral gap of the {Laplacian} on
  {Riemannian} manifolds.
\newblock {\em C. R. Math.}, 348(23):1319--1322, 2010.

\bibitem{zahm2019}
O.~Zahm, P.~Constantine, C.~Prieur, and Y.~Marzouk.
\newblock Gradient-based dimension reduction of multivariate vector-valued
  functions.
\newblock {\em SIAM J. Sci. Comput.}, 42(1):a534--a558, 2020.

\end{thebibliography}
	
	\bigskip
	\noindent
	(D. Heredia) UMR CNRS 5219,  \textsc{Institut de Math\'ematiques de
		Toulouse, Universit\'e de Toulouse, France}\par
	\textit{E-mail address:} \href{mailto:dheredia@insa-toulouse.fr}{mailto:dheredia(at)insa-toulouse.fr}\par
	\textit{URL:} \url{https://davidherediag.wordpress.com/}
	
\end{document}